\documentclass[11pt, reqno]{amsart}
\usepackage{a4wide}
\usepackage{wrapfig}
\usepackage{amsmath}
\usepackage{amsthm}
\usepackage{mathrsfs}
\usepackage{color}
\usepackage{xcolor}
\usepackage{graphicx}
\usepackage{amssymb}
\usepackage{ dsfont }
\usepackage{url}
\usepackage{multicol}
\usepackage{tikz}
\usepackage{amsmath}
\usepackage{fancyhdr}
\usetikzlibrary{matrix}
\usetikzlibrary{arrows}
\usepackage{subfig}
\usepackage{hyperref}
\usepackage{listings}

\usepackage{esint}

\usepackage{enumitem}
\usepackage{varwidth}

\allowdisplaybreaks

\title[On long-time behavior of solutions of the ZR/BR equation]
{On long-time behavior of solutions of the 
Zakharov-Rubenchik/Benney-Roskes system}

\author[M. E. Mart\'inez]{Mar\'ia E. Mart\'inez}
\address{Departamento de Ingenier\'ia Matem\'atica, 
Facultad de Ciencias F\'isicas y Matem\'aticas, Universidad de Chile
Beauchef 851, Santiago, Chile}
\email{maria.martinez.m@uchile.cl}

\author[J.M. Palacios]{Jos\'e M. Palacios}
\address{Institut Denis Poisson, Universit\'e de Tours, 
Universit\'e d'Orleans, CNRS, Parc Grandmont 37200, Tours, France}
\email{jose.palacios@lmpt.univ-tours.fr}

\newcommand{\be}{\begin{equation}}
\newcommand{\ee}{\end{equation}}
\newcommand{\bp}{\begin{proof}}
\newcommand{\ep}{\end{proof}}
\newcommand{\bel}{\begin{equation}\label}
\newcommand{\eeq}{\end{equation}}
\newcommand{\bea}{\begin{eqnarray}}
\newcommand{\eea}{\end{eqnarray}}
\newcommand{\bee}{\begin{eqnarray*}}
\newcommand{\eee}{\end{eqnarray*}}
\newcommand{\ben}{\begin{enumerate}}
\newcommand{\een}{\end{enumerate}}

\newcommand{\R}{\mathbb{R}}

\newcommand{\N}{\mathbb{N}}

\newcommand{\sech}{\operatorname{sech}}

\newcommand{\supp}{\operatorname{supp}}

\newtheorem{thm}{Theorem}[section]
\newtheorem{cor}[thm]{Corollary}
\newtheorem{lem}[thm]{Lemma}
\newtheorem{prop}[thm]{Proposition}

\theoremstyle{remark}
\newtheorem{rem}{Remark}[section]

\definecolor{codegreen}{rgb}{0,0.6,0}
\definecolor{codegray}{rgb}{0.5,0.5,0.5}
\definecolor{codepurple}{rgb}{0.58,0,0.82}
\definecolor{backcolour}{rgb}{0.95,0.95,0.92}

\lstdefinestyle{mystyle}{
	backgroundcolor=\color{backcolour},
	commentstyle=\color{codegreen},
	keywordstyle=\color{magenta},
	numberstyle=\tiny\color{codegray},
	stringstyle=\color{codepurple},
	basicstyle=\footnotesize,
	breakatwhitespace=false,
	breaklines=true,
	captionpos=b,
	keepspaces=true,
	numbers=left,
	numbersep=5pt,
	showspaces=false,
	showstringspaces=false,
	showtabs=false,
	tabsize=2
}
\numberwithin{equation}{section}

\usepackage{pgfplots}
\pgfplotsset{compat=newest}



\theoremstyle{definition}

\numberwithin{ej}{section}

\begin{document}


\begin{abstract}
We study decay properties for  solutions to the initial value problem associated with the one-dimensional Zakharov-Rubenchik/Benney-Roskes system. We prove time-integrability in growing compact intervals of size $t^{r}$, $r<2/3$, centered on some characteristic curves coming from the underlying transport equations associated with the ZR/BR system. Additionally, we prove decay to zero of the local energy-norm in so-called far-field regions. Our results are independent of the size of the initial data and do not require any parity condition. 
\end{abstract}

\maketitle

\medskip
\section{Introduction and main results}

\subsection{The model}
In this work we seek to show decay properties for solutions of the 
initial value problem (IVP) associated with the 
Zakharov-Rubenchik/Benney-Roskes (ZR/BR) system in one space dimension
\begin{align}
\label{zr_br_eq}
\begin{cases}
i\partial_t\psi+\omega \partial_x^2\psi=
\gamma \big(\eta -\tfrac{1}{2}\alpha \rho+q\vert\psi\vert^2\big)\psi,
\\ \theta\partial_t\rho+\partial_x\big(\eta-\alpha\rho\big)
=-\gamma \partial_x(\vert \psi\vert^2),
\\ \theta\partial_t \eta+\partial_x\big(\beta\rho-\alpha\eta\big)
=\tfrac{1}{2}\alpha\gamma\partial_x(\vert\psi\vert^2),
\\ \psi(0,x)=\psi_0(x),\ \rho(0,x)=\rho_0(x), \ \eta(0,x)=\eta_0(x).
\end{cases}
\end{align}
Here $\psi(t,x)$ denotes a complex-valued function, while $\rho(t,x)$ 
and $\eta(t,x)$ are both real-valued functions, and $t,x\in\R$. 
All Greek letters $(\omega,\alpha,\beta,\gamma,\theta)$ denote real 
parameters, and in the sequel we shall always assume that 
\[
\omega>0, \quad \beta>0, \quad \gamma>0, \quad \beta-\alpha^2>0,
\quad 0<\theta<1, \quad \hbox{and}\quad 
q:=\gamma+\dfrac{\alpha(\alpha\gamma-1)}{2(\beta-\alpha^2)}.
\]
Model \eqref{zr_br_eq} corresponds to the one-dimensional case of the most general system derived by Zakharov and Rubenchik \cite{ZaRu} to describe the interaction of spectrally
narrow high-frequency wave packets of small amplitude with low-frequency acoustic type
oscillations. This system was also independently found by Benney and Roskes \cite{BeRo} in the context of gravity waves, and in the $3$-dimensional case has the following form \begin{align}
\label{zr_br_eq_2}
\begin{cases}	
i \partial_t \psi +i v_g \partial_z \psi = 
-\frac{\omega''}{2}\partial_z^2 \psi - \frac{v_g}{2k}
\Delta_{\perp} \psi 
+ (q|\psi|^2+\beta\rho +\alpha \partial_z \eta)\psi, 
\\
\partial_t \rho +\rho_0 \Delta\eta 
+\alpha \partial_z |\psi|^2=0, 
\\
\partial_t \eta +\frac{c^2}{\rho_0}\rho+\beta |\psi|^2=0,
\end{cases}
\end{align}
where $\Delta_\perp=\partial_x^2 +\partial_y^2$. In this context, $\psi(t,x)$ stands for the amplitude of the carrying (high frequency) waves with wave number $k$, frequency $\omega=\omega(k)$ and $v_g =\omega'(k)$ stands for its group velocity. On the other hand, $\rho(t,x)$ and $\eta(t,x)$ correspond to the density fluctuation and the hydrodynamic potential respectively. 

\medskip

System \eqref{zr_br_eq_2} has also been derived in several other physical situations, such as for example, in the study of Alfv\'en waves (transverse oscillations of the magnetic fields) in the Magneto-Hydrodynamics equations (see for instance \cite{ChLaPaSu,PaSuSu}). Moreover, system \eqref{zr_br_eq_2} contains various important models as limiting cases, such as the
classical (scalar) Zakharov system and the Davey-Stewartson systems. We refer to \cite{Co} for a rigorous justification of the Zakharov limit (supersonic limit) of the ZR/BR system. However, the rigorous proof of the Davey-Stewartson limit from system \eqref{zr_br_eq_2} remains still open. 

\medskip

In the one dimensional case the situation is a little better understood. In fact, in this case we can also consider the adiabiatic limit, that is, to take $\theta\to 0$ in \eqref{zr_br_eq}, from where we can formally see  that $\rho(t,x)$ and $\eta(t,x)$ satisfy now the following relations
\[
\rho=-\frac{\gamma \alpha}{2(\beta -\alpha^2)}|\psi|^2, 
\quad 
\eta= - \gamma \frac{\beta-\alpha^2/2}{\beta-\alpha^2}
|\psi|^2.
\]
Then, we infer that 
the complex amplitude $\psi$ solves 
the cubic nonlinear Schr\"odinger equation
\[
i \partial_t \psi + \omega \partial_x^2 \psi = 
- \frac{\gamma \alpha}{3 (\beta- \alpha^2)}|\psi|^2 \psi.
\]
A rigorous justification of such limit (for well-prepared initial data) was proved by 
Oliveira in \cite{Ol2}.  Therefore, we can certainly see that the ZR/BR system is thus richer than those models.

\medskip

On the other hand, the ZR/BR system \eqref{zr_br_eq} and \eqref{zr_br_eq_2} posses a Hamiltonian 
structure \cite{ZaRu}, and hence, it follows (at least formally) 
that the energy of 
system \eqref{zr_br_eq} is conserved along the trajectory, which in the one-dimensional case can be written as
\begin{align*}
E\big(\psi(t),\rho(t),\eta(t)\big)&
:=\int_\R\Big(\omega\vert\psi_x\vert^2%
+\tfrac{\gamma q}{2}\vert\psi\vert^4+\tfrac{\beta}{2}\rho^2
+\tfrac{1}{2}\eta^2+\tfrac{\gamma}{2}(2\eta-\alpha\rho) \vert\psi\vert^2
-\alpha\rho\eta\Big)dx
\\ & =E(\psi_0,\rho_0,\eta_0).
\end{align*}
Moreover, the ZR/BR system \eqref{zr_br_eq} also conserves (formally) the mass and 
the momentum of the solution, which are given by the following relations (respectively)
\begin{align*}
M\big(\psi(t),\rho(t),\eta(t)\big)
&:=\int \vert \psi(t,x)\vert^2dx
=M(\psi_0,\rho_0,\eta_0), \quad \hbox{ and},
\\ P\big(\psi(t),\rho(t),\eta(t)\big)
&:=\mathrm{Im}\int_\R\psi\overline{\psi}_x
-\theta\int_\R \rho(t,x)\eta(t,x)dx=P(\psi_0,\rho_0,\eta_0).
\end{align*}
Additionally, related to these conservation laws, the ZR/BR 
system \eqref{zr_br_eq} is invariant under space-time translations, as well as invariant under phase rotations.

\medskip

Regarding the existence of solitary waves, in the case 
$\beta-\alpha^2>0$, 
$\gamma>0$ and $\theta< 1$, Oliveiro has proved in \cite{Ol} the existence and the orbital stability of solitary waves of the form
\begin{align}\label{soliton}
\big(\psi,\rho,\eta)(t,x):=\big(e^{i\lambda t}e^{icx/2\omega }
R(x-ct),a(c)\vert R(x-ct)\vert^2,b(c)\vert R(x-ct)\vert^2\big),
\end{align}
where $\lambda\in\R$, $c\geq 0$ and $R(\cdot)$ is an positive, even and exponentially decaying 
complex-valued function, while $a(c)$ and $b(c)$ are given by the following formulas 
\[
a(c):=-\frac{\gamma(\beta -\tfrac{\alpha}{2}(c\theta+\alpha))}
{\beta-(c\theta +\alpha)^2}, \quad b(c):=-\frac{\gamma(c\theta +\tfrac{1}{2}\alpha)}
{\beta-(c\theta+\alpha)^2}.
\]
In particular, the analysis carried out by Oliveiro shows that a necessary condition for these solitary waves to exists is that the following two inequalities must be satisfied
\begin{align}\label{condition_soliton_existence}
a(c)-\dfrac{\alpha}{2} b(c)+q<0 \quad \hbox{and} \quad \dfrac{c^2}{4\omega}-\lambda<0.
\end{align}
On the other hand, recently in \cite{LuMaSa} Luong \emph{et al.} studied the existence of the so-called bright and dark solitons for system \eqref{zr_br_eq}. They proved their existence under some conditions on the coefficients of the equations (similar to the one in \eqref{condition_soliton_existence}). Then, they used these solitons to construct line-solitons for the higher dimensional case. However, none of these solitons belongs to the energy space since they do not decay at $\pm \infty$ (see \cite{LuMaSa} for further details).

\medskip

Finally, concerning  the well-posedness for system 
\eqref{zr_br_eq},
Oliveiro \cite{Ol} proved local and global well-posedness 
for the one-dimensional case in 
$ H^2(\R)\times H^1(\R)\times H^1(\R)$. Later, Linares and 
Matheus \cite{LiMa} 
extended the result given by Oliveira showing local (and then 
global) well-posedness for inital data in the energy space 
$H^1(\R)\times L^2(\R)\times L^2(\R)$. Additionally, 
a polinomial bound for 
the growth of the $H^s$-norm of $\psi$ was stated in 
\cite{LiMa}. 
More specifically, they proved that,  for smooth initial data, solutions to system \eqref{zr_br_eq} satisfies the following property: 
 \[
\Vert \psi\Vert_{H^s(\R)}\lesssim 1+\vert t\vert^{(s-1)^+}.
\]
In fact, Linares and Matheus used this property to show that system \eqref{zr_br_eq} is globally well-posed in $H^k(\R)\times H^{k-\frac{1}{2}}(\R)\times H^{k-\frac{1}{2}}(\R)$ for all $k\geq0$. Moreover, regarding the higher dimensional cases, Ponce and Saut  \cite{PoSa} have proved 
that \eqref{zr_br_eq_2} is locally 
well posed in 
$H^s(\R^d)\times H^{s-\frac{1}{2}}(\R^d)\times H^{s+\frac{1}{2}}(\R^d)$, for 
$s>d/2$, where the space-dimension $d=2,3$. Lastly, we 
mention that Luong \emph{et al.} have recently proved the 
well-posedness (under some extra conditions) of system 
\eqref{zr_br_eq_2} in the background of a line-soliton 
\cite{LuMaSa}.

\subsection{Main results}

In the remainder of this work we focus in decay properties for general solutions of \eqref{zr_br_eq} in the energy space. Our first main result states that there exists two specific characteristic curves such that, along them, there is an additional time-integrability property on growing compact sets.

\begin{thm}\label{MT1}
Let $\upsilon_\pm:=\pm\theta^{-1}\big(\sqrt{\beta}\pm\alpha\big)$ fixed. Consider $(\psi, \rho, \eta )\in C(\R, H^1\times L^2\times L^2)$ to be any solution to system \eqref{zr_br_eq} emanating from an initial data $(\psi_0,\rho_0,\eta_0)\in H^1\times L^2\times L^2$. Then, for any $c\in\R_+$, the following inequality holds \[
\int_{0}^{+\infty}\frac{1}{\mu_*(t)}\int_{\Omega_\pm(t)}\vert\psi(t, x)\vert^2
dxdt<+\infty,
\]
where $\Omega_\pm(t):=\{x\in\R: \, -c\lambda(t)\leq x-\upsilon_\pm t \leq c\lambda(t)\}$, $\kappa:=10^{100}$ and
\begin{align*}
\lambda(t)&:=t^{2/3}\log\log^{-2/3}(\kappa+t) \quad \hbox{and}\quad \mu_*(t):=t\log(\kappa+t)\log\log(\kappa+t).
\end{align*}
Furthermore, we have the following scenarios:
\begin{enumerate}[leftmargin=9mm]
\item[$1$.] If $\pm\alpha< 0$, then, the following inequality holds
\begin{align*}
\int_{0}^{+\infty}\frac{1}{\mu_*(t)}\int_{\Omega_\pm(t)}\big(\vert\psi_x(t,x)\vert^2+\vert\psi(t, x)\vert^2 +\rho^2(t, x)+\eta^2(t, x)\big)
dxdt&<+\infty.
\end{align*}
In particular, we have that \[
\liminf_{t\to+\infty}\int_{\Omega_\pm(t)}\big(\vert\psi_x(t,x)\vert^2+\vert\psi(t, x)\vert^2 +\rho^2(t, x)+\eta^2(t, x)\big)
dx=0.
\]
\item[$2$.] If $\alpha=0$,  then, the following inequality holds \begin{align*}
\int_0^{+\infty}\dfrac{1}{\mu_*(t)}\int_{\Omega_0(t)}\Big(\vert\psi_x(t,x)\vert^2+\vert\psi(t,x)\vert^4+\eta^2(t,x)+\rho^2(t,x)\Big)dxdt<+\infty,
\end{align*}
where $\lambda$ and $\mu_*$ defined as above and $\Omega_0(t):=\{x\in\R: \, c\lambda(t)\leq\vert x\vert\leq C\lambda(t)\}$. In particular, the following is satisfied \[
\liminf_{t\to+\infty} \int_{\Omega_0(t)}\Big(\vert\psi_x(t,x)\vert^2+\vert\psi(t,x)\vert^4+\eta^2(t,x)+\rho^2(t,x)\Big)dx=0.
\]
\end{enumerate}

\end{thm}

\begin{rem}
It is important to notice that, as soon as $\alpha\neq 0$ we cannot deduce any time-integrability nor decay property on compacts sets centered at the origin. Of course, this is a consequence of (and consistent with) the existence of the standing-wave solution presented in \eqref{soliton}. On the other hand, when $\alpha=0$, condition  \eqref{condition_soliton_existence} does not allow standing-wave solutions to exists, more specifically, the first inequality in \eqref{condition_soliton_existence} is not satisfied when $c=\alpha=0$, and hence item $2$ is not contradictory with the existence of such family of solutions. 
\end{rem}

Our second main result states that, in the so-called far-field region, solutions (in the energy space) must decay to zero.
\begin{thm}\label{MT2}
	Let $(\psi, \rho, \eta )\in C(\R, H^1\times L^2\times L^2)$ 
	be any solution to system \eqref{zr_br_eq} emanating from 
	an 
	initial data 
	$(\psi_0,\rho_0,\eta_0)\in H^1\times L^2\times L^2$. Then, for any pair of constants $c_1,c_2>0$ the following properties holds:
	\begin{enumerate}[leftmargin=9mm]
	\item[$1$.]	
	 Consider any non-negative function $\zeta \in C^1(\R)$ satisfying that, there exists $\delta>0$ such that, for all $t>0$ it holds
	\[
	\zeta(t)\gtrsim t\log(\kappa+t)^{1+\delta} 
	\quad \text{ and } \quad
	\zeta'(t)
	\gtrsim \log(\kappa+t)^{\delta+1}.
	\]
	Then, setting $\Omega_\zeta(t):=\{x\in\R: \, c_1\zeta(t)\leq \vert x\vert\leq c_2\zeta(t)\}$, the following limit holds
	\begin{equation}\label{MT2Tesis1}
	\lim_{t\to \infty}
	\Vert\psi(t)\Vert_{L^2(\Omega_\zeta(t))}=0.
	\end{equation}
	\item[$2$.] Assume additionally that
	$(\psi, \rho, \eta )\in C(\R, H^2\times H^1\times H^1)$ is a solution emanating from an initial data $(\psi_0,\rho_0,\eta_0)\in H^2\times H^1\times H^1$. Then, for any non-negative $\zeta \in C^1(\R)$ satisfying that, there exists $\delta>0$ such that, for all $t>0$,
	\[
	\zeta(t)\gtrsim t^{2+\delta}
	\quad \text{ and } \quad
	\zeta'(t)
	\gtrsim t^{1+\delta},
	\]
	the 
	following decay for the local energy norm holds
	\[
	\lim_{t\to \infty}
	\left( \Vert\psi(t)\Vert_{H^1(\Omega_\zeta(t))}
	+\Vert\rho(t)\Vert_{L^2(\Omega_\zeta(t))}
	+\Vert\eta(t)\Vert_{L^2(\Omega_\zeta(t))}\right)=0.
	\] 
	\end{enumerate}

\end{thm}

\begin{rem}
Note that none of the above theorems require any smallness assumption in terms of the initial data $\Vert \psi_0\Vert_{H^1}\ll 1$. Moreover, they do not require any parity assumption either (as their counterparts founded in \cite{Ma,Ma2}), nor any extra decay hypotheses in terms of weighted Sobolev norms, such as 
$\Vert x\psi\Vert_{L^2}\ll 1$ for example.
\end{rem}

\begin{rem}\label{Parity_smallness_asumptions}
One important difference between the results above and those in 
\cite{Ma,Ma2} is that, in both of those works, the equations 
under study preserve the oddness of the initial data (for the Schr\"odinger component $\psi$), while 
system \eqref{zr_br_eq} does not. Hence, the analysis presented 
there assuming parity conditions on the initial data cannot be applied to system \eqref{zr_br_eq}.
\end{rem}

Finally, it is worth mentioning 
that the techniques involved in the proof of Theorems 
\ref{MT1} and \ref{MT2} have already been used before in some 
other contexts. We refer to \cite{Ma} for the use of some of 
these ideas in context of the one-dimensional Schr\"odinger 
equation, and to \cite{Ma2} for scalar Zakharov system (as 
well as the Klein-Gordon Zakharov system). On the other hand, 
for other type of systems that have served us for motivations 
we refer to \cite{KwMu,LiMe,MuPoSa}. However, as previously 
described, 
system \eqref{zr_br_eq} has some important 
differences with respect to the above cases 
(see Remark \ref{Parity_smallness_asumptions}), 
what does not allow us to apply the same ideas. 
In particular, the presence of some transport equations in 
\eqref{zr_br_eq} breaks the symmetry properties used in 
previous works to study Schr\"odinger-type equations/systems 
with these specific techniques.

\section{Preliminary lemmas}

\subsection{Virial identities}
In this section we seek to establish the key virial identities required 
in our analysis. In order to do that we consider the following weight 
function
\[
\Phi(x):=\tanh(x), \quad \hbox{ and hence } \quad  \Phi'= \sech^2(x).
\]
Additionally, we consider time-dependent scaling functions 
$\lambda_1(t)$, $\lambda_2(t)$ and $\mu(t)$ given by 
\begin{align}\label{def_scalings_time}
\lambda_1(t)&:=(\kappa+t)^{2/3}\log\log^{-2/3}(\kappa+t),\nonumber
\\ \lambda_2(t)&:=(\kappa +t)^{2/3}\log\log^{1/3}(\kappa+t),
\\ \mu(t)&:=(\kappa+t)^{1/3}\log(\kappa+t)\log\log^{5/3}(\kappa+t),\nonumber
\end{align}
where $\kappa:=10^{100}$. The role of $\mu(t)$ and $\lambda_i(t)$ is to provide some extra time-decay 
so that we can somehow neglect bad terms (with no sign) that prevent us to conclude the properties claimed in the above theorems. Additionally, one can think of $\lambda_1(t)$ as the rate of growth of the set $\Omega_\pm(t)$ and $\Omega_0(t)$ defined in Theorem \ref{MT1}. The key idea for considering exactly these definitions for 
$\mu(t)$ and $\lambda_i(t)$ is that 
\[
\dfrac{1}{\mu(t)\lambda_2(t)},\,\dfrac{\lambda_i'(t)}{\mu(t)\lambda_i(t)},\, 
\dfrac{\mu'(t)}{\mu^2(t)}\in L^1(\R_+), 
\quad \hbox{ while } 
\quad \dfrac{1}{\mu(t)\lambda_1(t)}\notin L^1(\R_+).
\]
In the sequel we shall exploit these two properties. Moreover, for the 
sake of simplicity we introduce the following useful notation 
\begin{align}\label{def_upsilon_pm}
\upsilon_-:=-\theta^{-1}\big(\sqrt{\beta}-\alpha\big) 
\quad \hbox{ and }\quad \upsilon_+ 
:=\theta^{-1}\big(\sqrt{\beta}+\alpha\big).
\end{align}
Then, with all of the above notations, we define the \emph{modified mean} functionals $\mathcal{J}_1(t)$ 
and $\mathcal{J}_2(t)$, adapted to the curves $x-\upsilon_\pm t$, which are given by \begin{align*}
\mathcal{J}_1(t)&:=
\dfrac{\theta}{\mu(t)}\int_\R\Big(\sqrt{\beta}\rho
\big(t,x-\upsilon_- t\big)+\eta\big(t,x-\upsilon_- t\big)\Big)
\Phi\left(\dfrac{x}{\lambda_1(t)}\right)
\Phi'\left(\dfrac{x}{\lambda_2(t)}\right)dx,
\\ 
\mathcal{J}_2(t)&
:=\dfrac{\theta}{\mu(t)}\int_\R \Big(\sqrt{\beta}\rho\big(t,x-\upsilon_+t\big)-\eta\big(t,x-\upsilon_+t\big)\Big)
\Phi\left(\dfrac{x}{\lambda_1(t)}\right)
\Phi'\left(\dfrac{x}{\lambda_2(t)}\right)dx.
\end{align*}
The reason why evaluating solutions on these translated points 
$x-\upsilon_\pm t$ is to be able to take advantage of the 
characteristics of the underlying transport equations associated to 
\eqref{zr_br_eq}. Moreover, another key quantity that shall play a 
fundamental role in our proof is the \emph{modified momentum} functional 
$\mathcal{I}(t)$, which is given by
\[
\mathcal{I}(t):=
\dfrac{1}{\mu(t)}\mathrm{Im}\int_\R\psi(t,x)\overline{\psi_x}(t,x)
\Phi\left(\dfrac{x}{\lambda_1(t)}\right)dx-\dfrac{\theta}{\mu(t)}
\int_\R\rho(t,x)\eta(t,x)\Phi\left(\dfrac{x}{\lambda_1(t)}\right)dx.
\]
For the sake of simplicity and the clarity of computations, we split the 
previous functional into two parts, namely,
\begin{align*}
\mathcal{I}_1(t)
&:=\dfrac{1}{\mu(t)}\mathrm{Im}\int_\R\psi(t,x)\overline{\psi_x}(t,x)
\Phi\left(\dfrac{x}{\lambda_1(t)}\right)dx,
\\ \mathcal{I}_2(t)&:=\dfrac{\theta}{\mu(t)}
\int_\R\rho(t,x)\eta(t,x)\Phi\left(\dfrac{x}{\lambda_1(t)}\right)dx.
\end{align*}
We anticipate that, thanks to the explicit form of $\Phi$ and
 the conservation of the momentum and the energy, 
 if $(\psi,\rho,\eta)$ is a solution to system 
\eqref{zr_br_eq} belonging to the class 
$C(\R,H^1\times L^2\times L^2)$, 
then, all the modified functionals above 
$\mathcal{J}_1(t)$, $\mathcal{J}_2(t)$ and $\mathcal{I}(t)$ 
are well defined 
for all times $t\in\R$ 
(see Lemma \ref{uniform_bound_mod_functionals} 
for further details). 

\medskip

The following three lemmas give us the first basic virial 
identities satisfied by the modified functionals $\mathcal{J}_1(t)$, $\mathcal{J}_2(t)$ and $\mathcal{I}(t)$ above.
\begin{lem}\label{virial_identity_lemma}
	Let $(\psi,\rho,\eta)\in C(\R,H^1\times L^2\times L^2)$ be any 
	solution to system \eqref{zr_br_eq}. Then, for all $t\in\R$, the 
	following identity holds
	\begin{align}\label{di_dt}
	-\dfrac{d}{dt}\mathcal{I}(t)&=\dfrac{2\omega}{\mu\lambda_1}
	\int \vert\psi_x\vert^2\Phi'-\dfrac{\omega}{2\mu\lambda_1^3}
	\int \vert\psi\vert^2\Phi'''+\dfrac{1}{2\mu\lambda_1 }
	\int \eta^2\Phi'+\dfrac{\beta}{2\mu\lambda_1 }
	\int \rho^2\Phi' \nonumber
	\\ & \quad  +\dfrac{3\gamma q}{4\mu\lambda_1}
	\int \vert\psi\vert^4\Phi'
	-\dfrac{\alpha}{\mu\lambda_1 }\int \rho\eta\Phi'
	+\dfrac{\gamma}{\mu\lambda_1 }\int 
	\big(\eta-\tfrac{\alpha}{2}\rho\big)\vert\psi\vert^2\Phi'
	-\dfrac{\mu'}{\mu^2}\int\rho\eta\Phi
	\\ &\quad 
	+\dfrac{\mu'}{\mu^2}\mathrm{Im}\int\psi\overline{\psi_x}\Phi
	+\dfrac{\lambda_1'}{\mu\lambda_1}\mathrm{Im}
	\int\left(\dfrac{x}{\lambda_1}\right)\psi\overline{\psi}_x\Phi'
	-\dfrac{\lambda_1'}{\mu\lambda_1 }\int\left(\dfrac{x}{\lambda_1}
	\right)\rho\eta\Phi'. \nonumber
	\end{align}
\end{lem}

\begin{proof}
	The proof is somehow straightforward and follows 
	from direct computations; we shall only proceed formally. 
	Notice that the following reasoning can be made 
	rigorously by standards
	approximation and density arguments. 
	
	\medskip

	\par Directly 
	differentiating the definition of the functional 
	$\mathcal{I}_1$, using system \eqref{zr_br_eq} and 
	performing several integration by parts we obtain
	\begin{align*}
	\dfrac{d}{dt}\mathcal{I}_1(t)&=\dfrac{1}{\mu}\mathrm{Im}\int \psi_t\overline{\psi_x}\Phi+\dfrac{1}{\mu}\mathrm{Im}\int \psi\overline{\psi_{tx}}\Phi-\dfrac{\lambda'}{\mu\lambda}\mathrm{Im}\int \left(\dfrac{x}{\lambda}\right)\psi\overline{\psi_x}\Phi'-\dfrac{\mu'}{\mu^2}\mathrm{Im}\int\psi\overline{\psi_x}\Phi
	\\ & = \dfrac{2}{\mu}\mathrm{Im}\int \psi_t\overline{\psi}_x\Phi  -  \dfrac{1}{\mu\lambda}\mathrm{Im}\int \psi\overline{\psi_t}\Phi'-\dfrac{\lambda'}{\mu\lambda}\mathrm{Im}\int \left(\dfrac{x}{\lambda}\right)\psi\overline{\psi_x}\Phi'-\dfrac{\mu'}{\mu^2}\mathrm{Im}\int\psi\overline{\psi_x}\Phi
	\\ & = \dfrac{2}{\mu}\mathrm{Re} \int \Big( \omega  \psi_{xx}-\gamma \big(\eta-\tfrac{\alpha}{2}\rho+q\vert\psi\vert^2\big)\psi\Big)\overline{\psi_x}\Phi-\dfrac{\mu'}{\mu^2}\mathrm{Im}\int\psi\overline{\psi_x}\Phi
	\\ & \qquad +\dfrac{1}{\mu\lambda}\mathrm{Re}\int \Big(\omega  \overline{\psi_{xx}}-\gamma \big(\eta-\tfrac{\alpha}{2}\rho+q\vert\psi\vert^2\big)\overline{\psi}\Big)\psi\Phi'-\dfrac{\lambda'}{\mu\lambda}\mathrm{Im}\int \left(\dfrac{x}{\lambda}\right)\psi\overline{\psi_x}\Phi'
	\\ & = -\dfrac{2\omega}{\mu\lambda}\int \vert\psi_x\vert^2\Phi'+\dfrac{\gamma}{\mu\lambda}\int \Big(\eta-\tfrac{\alpha}{2}\rho+{\tfrac{q}{4}\vert\psi\vert^2}\Big)\vert\psi\vert^2\Phi'-\dfrac{\mu'}{\mu^2}\mathrm{Im}\int\psi\overline{\psi_x}\Phi
	\\ & \qquad +\dfrac{\gamma}{\mu}\int \big(\eta_x-\tfrac{\alpha}{2}\rho_x\big)\vert\psi\vert^2\Phi+\dfrac{\omega}{2\mu\lambda^3}\int \vert\psi\vert^2\Phi''' -\dfrac{\lambda'}{\mu\lambda}\mathrm{Im}\int\left(\dfrac{x}{\lambda}\right)\psi\overline{\psi}_x\Phi'
	\\ & \qquad -\dfrac{\gamma}{\mu\lambda}\int \big(\eta-\tfrac{\alpha}{2}\rho+q\vert\psi\vert^2\big)\vert\psi\vert^2\Phi'
	\\ & = -\dfrac{2\omega}{\mu\lambda}\int \vert\psi_x\vert^2\Phi'-\dfrac{\mu'}{\mu^2}\mathrm{Im}\int\psi\overline{\psi_x}\Phi+\dfrac{\gamma}{\mu}\int \big(\eta_x-\tfrac{\alpha}{2}\rho_x\big)\vert\psi\vert^2\Phi
	\\ & \qquad +\dfrac{\omega}{2\mu\lambda^3}\int \vert\psi\vert^2\Phi''' -\dfrac{3\gamma q}{4\mu\lambda}\int \vert\psi\vert^4\Phi'-\dfrac{\lambda'}{\mu\lambda}\mathrm{Im}\int\left(\dfrac{x}{\lambda}\right)\psi\overline{\psi}_x\Phi'
	\end{align*}
Now we compute the time-derivative of the second functional $\mathcal{I}_2$.
In fact, by direct differentiation again, using system \eqref{zr_br_eq} and 
performing several integration by parts we get
	\begin{align*}
	\dfrac{d}{dt}\mathcal{I}_2&=\dfrac{1}{\mu}\int \rho_t\eta\Phi+\dfrac{1}{\mu}\int \rho\eta_t\Phi-\dfrac{\lambda '}{\mu\lambda }\int\left(\dfrac{x}{\lambda }\right)\rho\eta\Phi'-\dfrac{\mu'}{\mu^2}\int\rho\eta\Phi
	\\ & = \dfrac{1}{2\mu\lambda }\int \eta^2\Phi'+\dfrac{\alpha}{\mu}\int \rho_x\eta\Phi+\dfrac{\gamma}{\mu}\int \eta_x\vert\psi\vert^2\Phi+\dfrac{\gamma}{\mu\lambda }\int \eta\vert\psi\vert^2\Phi'
	\\ & \qquad +\dfrac{\beta}{2\mu\lambda }\int \rho^2\Phi'+\dfrac{\alpha}{\mu}\int \rho\eta_x\Phi-\dfrac{\alpha\gamma}{2\mu}\int \rho_x\vert\psi\vert^2\Phi-\dfrac{\alpha\gamma}{2\mu\lambda }\int \rho\vert\psi\vert^2\Phi'
	\\ &= \dfrac{1}{2\mu\lambda }\int \eta^2\Phi'-\dfrac{\alpha}{\mu\lambda }\int \rho\eta\Phi'+\dfrac{\gamma}{\mu}\int \eta_x\vert\psi\vert^2\Phi+\dfrac{\gamma}{\mu\lambda }\int \eta\vert\psi\vert^2\Phi'-\dfrac{\mu'}{\mu^2}\int\rho\eta\Phi
	\\ & \qquad +\dfrac{\beta}{2\mu\lambda }\int \rho^2\Phi'-\dfrac{\alpha\gamma}{2\mu}\int \rho_x\vert\psi\vert^2\Phi-\dfrac{\alpha\gamma}{2\mu\lambda }\int \rho\vert\psi\vert^2\Phi'-\dfrac{\lambda '}{\mu\lambda }\int\left(\dfrac{x}{\lambda }\right)\rho\eta\Phi'.
	\end{align*}
	Hence, gathering both previous identities we conclude the desired result.
\end{proof}

\begin{lem}\label{virial_j1}
Let $(\psi,\rho,\eta)\in C(\R,H^1\times L^2\times L^2)$ 
be any 
solution to system \eqref{zr_br_eq}. 
Then, for all $t\in\R$, 
the following identities hold:
\begin{align*}
\mathit{(1)} \ \dfrac{d}{dt}\mathcal{J}_1(t)&=%
\dfrac{\gamma(2\sqrt{\beta}-\alpha)}{\mu(t)\lambda_1(t)}\int \big\vert\psi\big (t,x-\upsilon_- t\big)\big\vert^2\Phi'\left(\dfrac{x}{\lambda_1(t)}\right)\Phi'\left(\dfrac{x}{\lambda_2(t)}\right)
\\ & \quad +\dfrac{\gamma (2\sqrt{\beta}-\alpha)}{\mu(t)\lambda_2(t)}\int \big\vert\psi\big (t,x-\upsilon_- t\big)\big\vert^2\Phi\left(\dfrac{x}{\lambda_1(t)}\right)\Phi''\left(\dfrac{x}{\lambda_2(t)}\right)
\\ & \quad -\dfrac{\theta\mu'(t)}{\mu^2(t)}\int \big(\sqrt{\beta}\rho-\eta\big)(t,x-\upsilon_- t)\Phi\left(\dfrac{x}{\lambda_1(t)}\right)\Phi'\left(\dfrac{x}{\lambda_2(t)}\right)
\\ & \quad -\dfrac{\theta\lambda_1'(t)}{\mu(t)\lambda_1(t)}\int \left(\dfrac{x}{\lambda_1(t)}\right)\big(\sqrt{\beta}\rho-\eta\big)(t,x-\upsilon_- t)\Phi'\left(\dfrac{x}{\lambda_1(t)}\right)\Phi'\left(\dfrac{x}{\lambda_2(t)}\right)
\\ & \quad -\dfrac{\theta\lambda_2'(t)}{\mu(t)\lambda_2(t)}\int \left(\dfrac{x}{\lambda_2(t)}\right)\big(\sqrt{\beta}\rho-\eta\big)(t,x-\upsilon_- t)\Phi\left(\dfrac{x}{\lambda_1(t)}\right)\Phi''\left(\dfrac{x}{\lambda_2(t)}\right).
\end{align*}
\begin{align*}
	\mathit{(2)} \ 
	\dfrac{d}{dt}\mathcal{J}_2(t)&=%
	\dfrac{\gamma(\sqrt{\beta}-\alpha)}{\mu(t)\lambda_1(t)}\int _\R \big\vert\psi\big (t,x-\upsilon_+ t\big)\big\vert^2\Phi'\left(\dfrac{x}{\lambda_1(t)}\right)\Phi'\left(\dfrac{x}{\lambda_2(t)}\right)
	\\ & \quad +\dfrac{\gamma (\sqrt{\beta}-\alpha)}{\mu(t)\lambda_2(t)}\int_\R \big\vert\psi\big (t,x-\upsilon_+ t\big)\big\vert^2\Phi\left(\dfrac{x}{\lambda_1(t)}\right)\Phi''\left(\dfrac{x}{\lambda_2(t)}\right)
	\\ & \quad -\dfrac{\theta\mu'(t)}{\mu^2(t)}\int \big(\sqrt{\beta}\rho-\eta\big)(t,x-\upsilon_+ t)\Phi\left(\dfrac{x}{\lambda_1(t)}\right)\Phi'\left(\dfrac{x}{\lambda_2(t)}\right)
	\\ & \quad -\dfrac{\theta\lambda_1'(t)}{\mu(t)\lambda_1(t)}\int \left(\dfrac{x}{\lambda_1(t)}\right)\big(\sqrt{\beta}\rho-\eta\big)(t,x-\upsilon_+ t)\Phi'\left(\dfrac{x}{\lambda_1(t)}\right)\Phi'\left(\dfrac{x}{\lambda_2(t)}\right)
	\\ & \quad -\dfrac{\theta\lambda_2'(t)}{\mu(t)\lambda_2(t)}\int \left(\dfrac{x}{\lambda_2(t)}\right)\big(\sqrt{\beta}\rho-\eta\big)(t,x-\upsilon_+ t)\Phi\left(\dfrac{x}{\lambda_1(t)}\right)\Phi''\left(\dfrac{x}{\lambda_2(t)}\right).
	\end{align*}
\end{lem}

\begin{proof}
Similarly as in the previous lemma, we proceed by a direct 
computation. For the sake of simplicity we shall only write 
$\rho$, $\eta$, $\Phi_1$, ommiting their arguments. 
Indeed, taking the time derivative of the functional, using 
system \eqref{zr_br_eq} and performing some integration by 
parts we obtain
\begin{align*}
\dfrac{d}{dt}\mathcal{J}_1(t)&=\dfrac{\sqrt{\beta}\theta}{\mu}\int \rho_t\Phi_1\Phi_2'+ \dfrac{\beta-\alpha\sqrt{\beta}}{\mu}\int\rho_x\Phi_1\Phi_2'-\dfrac{\theta\sqrt{\beta}\mu'}{\mu^2}\int \rho\Phi_1\Phi_2'+\dfrac{\theta}{\mu}\int\eta_t\Phi_1\Phi_2'
\\ & \quad +\dfrac{\sqrt{\beta}-\alpha}{\mu}\int\eta_x\Phi_1\Phi_2'-\dfrac{\theta\mu'}{\mu^2}\int\eta\Phi_1\Phi_2'-\dfrac{\sqrt{\beta}\theta\lambda_1'}{\mu\lambda_1}\int \left(\dfrac{x}{\lambda_1}\right)\rho\Phi_1'\Phi_2'
\\ & \quad -\dfrac{\theta\lambda_1'}{\mu\lambda_1}\int \left(\dfrac{x}{\lambda_1}\right)\eta\Phi_1'\Phi_2'-\dfrac{\sqrt{\beta}\theta\lambda_2'}{\mu\lambda_2}\int \left(\dfrac{x}{\lambda_2}\right)\rho\Phi_1\Phi_2''-\dfrac{\theta\lambda_2'}{\mu\lambda_2}\int \left(\dfrac{x}{\lambda_2}\right)\eta\Phi_1\Phi_2''
\\ & = \dfrac{\alpha\sqrt{\beta}}{\mu}\int \rho_x\Phi_1\Phi_2'-\dfrac{\sqrt{\beta}}{\mu}\int\eta_x\Phi_1\Phi_2' -\dfrac{\gamma \sqrt{\beta}}{\mu}\int (\vert\psi\vert^2)_x\Phi_1\Phi_2'+\dfrac{\alpha}{\mu} \int\eta_x\Phi_1\Phi_2'
\\ & \quad +\dfrac{\beta-\sqrt{\beta}\alpha}{\mu}\int\rho_x\Phi_1\Phi_2'-\dfrac{\sqrt{\beta}\theta\mu'}{\mu^2}\int \rho\Phi_1\Phi_2'-\dfrac{\beta}{\mu} \int\rho_x \Phi_1\Phi_2'-\dfrac{\theta\mu'}{\mu^2}\int\eta\Phi_1\Phi_2'
\\ & \quad +\dfrac{\alpha\gamma}{2\mu}\int(\vert\psi\vert^2)_x \Phi_1\Phi_2'+\dfrac{\sqrt{\beta}-\alpha}{\mu}\int\eta_x\Phi_1\Phi_2'-\dfrac{\sqrt\beta\theta\lambda_1'}{\mu\lambda_1}\int \left(\dfrac{x}{\lambda_1}\right)\rho\Phi_1'\Phi_2'
\\ & \quad -\dfrac{\theta\lambda_1'}{\mu\lambda_1}\int \left(\dfrac{x}{\lambda_1}\right)\eta\Phi_1'\Phi_2'-\dfrac{\sqrt{\beta}\theta\lambda_2'}{\mu\lambda_2}\int \left(\dfrac{x}{\lambda_2}\right)\rho\Phi_1\Phi_2''-\dfrac{\theta\lambda_2'}{\mu\lambda_2}\int \left(\dfrac{x}{\lambda_2}\right)\eta\Phi_1\Phi_2''
\\ & = \dfrac{\gamma \sqrt{\beta}}{\mu\lambda_1}\int \vert\psi\vert^2\Phi_1'\Phi_2'+\dfrac{\gamma \sqrt{\beta}}{\mu\lambda_2}\int \vert\psi\vert^2\Phi_1\Phi_2''-\dfrac{\sqrt{\beta}\theta\mu'}{\mu^2}\int \rho\Phi_1\Phi_2'-\dfrac{\theta\mu'}{\mu^2}\int\eta\Phi_1\Phi_2'
\\ & \quad -\dfrac{\alpha\gamma}{2\mu\lambda_1}\int\vert\psi\vert^2\Phi_1'\Phi_2'-\dfrac{\alpha\gamma}{2\mu\lambda_2}\int\vert\psi\vert^2\Phi_1\Phi_2''-\dfrac{\sqrt\beta\theta\lambda_1'}{\mu\lambda_1}\int \left(\dfrac{x}{\lambda_1}\right)\rho\Phi_1'\Phi_2'
\\ & \quad -\dfrac{\theta\lambda_1'}{\mu\lambda_1}\int \left(\dfrac{x}{\lambda_1}\right)\eta\Phi_1'\Phi_2'-\dfrac{\sqrt{\beta}\theta\lambda_2'}{\mu\lambda_2}\int \left(\dfrac{x}{\lambda_2}\right)\rho\Phi_1\Phi_2''-\dfrac{\theta\lambda_2'}{\mu\lambda_2}\int \left(\dfrac{x}{\lambda_2}\right)\eta\Phi_1\Phi_2''.
\end{align*}
Thus, by gathering terms we conclude the proof of the lemma. The proof of the second formula follows the same arguments. 
Hence, we omit it. 
\end{proof}

\begin{rem}
We emphasize that non of these Virial Lemmas require the explicit definition of $\Phi$ nor the one for the scaling functions $\lambda_i$ and $\mu$ that we gave at the beginning of this section. In fact, in Section \ref{farfield_sec} we shall exploit \eqref{di_dt} for completely different definitions of $\Phi$, $\lambda$ and $\mu$ (as soon as all quantities are well-defined). However, unless stated otherwise, throughout the proof of Theorem \ref{MT1} we shall always assume that we are referring to the functions defined at the beginning of this section.
\end{rem}

\subsection{Uniform boundedness of the energy norm}

The following lemma is a direct consequence of the conservation laws and give us the time-uniform boundedness of the $H^1\times L^2\times L^2$-norm.
\begin{lem}\label{uniform_bound_energy}
	Let $(\psi,\rho,\eta)\in C(\R,H^1\times L^2\times L^2)$ be a global solution emanating from an initial data $(\psi_0,\rho_0,\eta_0)\in H^1\times L^2\times L^2$. Then, there exists a constant $C\in\R_+$, depending only on the norm of the initial data, such that the following global bound holds
	\[
	\Vert \psi(t)\Vert_{H^1}^2+\Vert \rho(t)\Vert_{L^2}^2+\Vert \eta(t)\Vert_{L^2}^2\leq C, \qquad \forall t\in\R.
	\]
\end{lem}

\begin{proof}
	The idea of the proof is to use the conservation of the energy, re-constructing such conserved quantity from the energy norm.
	In fact, first of all let us recall that from the conservation of the energy we have \begin{align}\label{re_energy}
	\int_\R\Big(\omega\vert\psi_x\vert^2%
	+\tfrac{\beta}{2}\rho^2+\tfrac{1}{2}\eta^2+\tfrac{\gamma q}{2}\vert\psi\vert^4+\tfrac{\gamma}{2}(2\eta-\alpha\rho) \vert\psi\vert^2-\alpha\rho\eta\Big)dx=E(0).
	\end{align}
	Hence, essentially we have to show that we can control the last three addends with the first three of them.
	Specifically, taking advantage of the conservation of both mass and energy, we would like to find appropriate constants $c>0$, $a, b \in \R$
	such that
	\[
	\Vert \psi(t)\Vert_{H^1}^2+
	\Vert \rho(t)\Vert_{L^2}^2+ \Vert \eta(t)\Vert_{L^2}^2
	\le c \left( E(0)^a+ M(0)^b\right).
	\]

	\par
	First, we notice that to control the crossed term $\rho\eta$, it is enough to use Young inequality for products, from where we get
	\begin{align}\label{bound_crossed}
	\alpha\int_\R \rho(t,x)\eta(t,x)dx\leq \dfrac{\beta+\alpha^2}{4}\int_\R\rho^2(t,x)dx +\dfrac{\alpha^2}{2(\beta+\alpha^2)}\int_R\eta^2(t,x)dx.
	\end{align}
	Then, gathering the corresponding quadratic terms with respect to $(\rho,\eta)$ appearing in the energy, we have
	\[
	\int_\R\Big(\tfrac{\beta}{2}\rho^2(t,x)+\tfrac{1}{2}\eta^2(t,x)-\alpha\rho(t,x)\eta(t,x)\Big)dx\geq \dfrac{\beta-\alpha^2}{4}\Vert \rho(t)\Vert_{L^2}^2+\dfrac{\beta}{2(\beta+\alpha^2)}\Vert\eta(t)\Vert_{L^2}^2.
	\]
	%

	\par
	We continue by bounding the contribution of the 
	$L^4$-norm of $\psi(t)$. Indeed, by using 
	Gagliardo-Nirenberg interpolation inequality, 
	as well as Young inequality in the resulting 
	right-hand side, we obtain
	\[
	\int_\R \vert\psi(t,x)\vert^4 dx
	\leq   \Vert \psi(t)\Vert_{H^1}\Vert 
	\psi(t)\Vert_{L^2}^3
	\leq
	\varepsilon\Vert \psi_x(t)\Vert_{L^2}^2
	+\varepsilon{M}(0)+\dfrac{1}{\varepsilon}{M}(0)^3.
	\]
	Once again, due to the conservation of mass, it is 
	enough to choose $\varepsilon\in(0,1)$ sufficiently 
	small so that we can absorb $\varepsilon\Vert 
	\psi_x(t)\Vert_{L^2}^2$ by using the first term in 
	\eqref{re_energy}. Finally, it only remains to bound 
	\[
	\dfrac{\gamma}{2}\int_\R\big(2\eta(t,x)
	-\alpha\rho(t,x)\big)\vert\psi(t,x)\vert^2dx.
	\]
	However, notice that this term can be controlled by 
	the previous ones. In fact, we have 
	\begin{align*}
	\dfrac{\gamma}{2}\int_\R\big(2\eta(t,x)
	-\alpha\rho(t,x)\big)\vert\psi(t,x)\vert^2dx & 
	\leq \dfrac{\beta}{16(\beta+\alpha^2)}\Vert 
	\eta(t)\Vert_{L^2}^2
	+\dfrac{\beta-\alpha^2}{16}\Vert\rho(t)
	\Vert_{L^2}^2
	\\ & \qquad + 
	\left(\dfrac{\gamma^2(\beta+\alpha^2)}{8\beta}
	+\dfrac{2\alpha^2\gamma^2}{\beta-\alpha^2}
	\right)\Vert\psi(t)\Vert_{L^4}^4.
	\end{align*}
	Therefore, gathering all the above estimates we 
	conclude the proof of the lemma.
\end{proof}
As a consequence of the previous lemma, we conclude the 
uniform boundedness of all the modified functionals.
\begin{cor}\label{uniform_bound_mod_functionals}
Let $(\psi, \rho, \eta )\in C(\R, H^1\times L^2\times L^2)$ be any solution to system \eqref{zr_br_eq} emanating from an initial data $(\psi_0,\rho_0,\eta_0)\in H^1\times L^2\times L^2$. Consider 
$\lambda_1(t)$, $\lambda_2$ and $\mu(t)$ defined as in 
\eqref{def_scalings_time}. Then, the following bound holds
\begin{align*}
\sup_{t\in(0,+\infty)}\Big(\big\vert\mathcal{J}_1(t)
\big\vert+\big\vert\mathcal{J}_2(t)\big\vert
+\big\vert\mathcal{I}_1(t)\big\vert
+\big\vert\mathcal{I}_2(t)\big\vert\Big)<+\infty.
\end{align*}
\end{cor}

\begin{proof}
First of all, notice that the time-uniform boundedness 
of $\mathcal{I}_1$ and $\mathcal{I}_2$ follows directly from 
H\"older inequality as well as the previous Lemma. In the 
same fashion, to bound $\mathcal{J}_1$ we proceed by H\"older 
inequality. However, since $\mathcal{J}_1$ is of order $1$ in 
$(\rho,\eta)$, in this case we obtain 
\begin{align}\label{mal_bound}
\vert\mathcal{J}_1(t)\vert\lesssim 
\dfrac{1}{\mu(t)}\Vert\rho(t)+\eta(t)\Vert_{L^2}\Vert 
\Phi\Vert_{L^\infty}\left\Vert 
\Phi'\big(\tfrac{\cdot}{\lambda_2(t)}\big)\right\Vert_{L^2} 
\lesssim \dfrac{\lambda_2^{1/2}(t)}{\mu(t)}<C,
\end{align}
where $C>0$ only depends on the initial data 
$(\psi_0,\rho_0,\eta_0)$. 

\par To conclude, we notice that the same procedure also 
provides a time-uniform bound for $\mathcal{J}_2(t)$. The 
proof is complete.
\end{proof}
\begin{rem}
Inequality \eqref{mal_bound} is precisely the condition 
that does not allow us to choose $\lambda_1(t)$ growing any 
faster. In particular, this is the reason why we cannot 
choose $\lambda_1(t)=t^{1^-}$,  for example.
\end{rem}

\section{Proof of Theorem \ref{MT1}}

\subsection{Time integrability of $\vert\psi\vert^2$} 
In this section we seek to use the previously found virial 
identities to prove the time integrability of the solution. 
In order to do that, we split the analysis in several steps. 
First, we shall show the time integrability 
(in the region given in Theorem \ref{MT1}) only for 
$\vert\psi(t,x)\vert^2$, which is proved in the following 
proposition. 
\begin{prop}\label{integrability_l2_prop}
Let $(\psi, \rho, \eta )\in C(\R, H^1\times L^2\times L^2)$ be any solution to system \eqref{zr_br_eq} emanating from an initial data $(\psi_0,\rho_0,\eta_0)\in H^1\times L^2\times L^2$. Then, for $\lambda_1(t)$, $\lambda_2(t)$, $\mu(t)$ and $\upsilon_\pm$ defined as in \eqref{def_scalings_time}-\eqref{def_upsilon_pm}, the following inequality holds
\begin{align}\label{time_integrability_l2}
\int_{0}^{+\infty}\frac{1}{\mu(t)\lambda_1(t)}\int_{\R}\big\vert\psi(t, x-\upsilon_\pm t)\big\vert^2 \sech^4\left(\frac x {\lambda_1(t)}\right)
dxdt&<+\infty.
\end{align}
\end{prop}

\begin{proof}
Let us first consider the case 
of $\upsilon_-$. The case for $\upsilon_+$ 
follows from the same bounds up to trivial modifications. 
Indeed, we define
\begin{align*}
\mathcal{F}_-(t)&:= \dfrac{d}{dt}\mathcal{J}_1(t)-\dfrac{\gamma (2\sqrt{\beta}-\alpha)}{\mu(t)\lambda_2(t)}\int \big\vert\psi\big (t,x-\upsilon_- t\big)\big\vert^2\Phi\left(\dfrac{x}{\lambda_1(t)}\right)\Phi''\left(\dfrac{x}{\lambda_2(t)}\right)
\\ & \quad \,+\dfrac{\theta\mu'(t)}{\mu^2(t)}\int \big(\sqrt{\beta}\rho-\eta\big)(t,x-\upsilon_- t)\Phi\left(\dfrac{x}{\lambda_1(t)}\right)\Phi'\left(\dfrac{x}{\lambda_2(t)}\right)
\\ & \quad \,+\dfrac{\theta\lambda_1'(t)}{\mu(t)\lambda_1(t)}\int \left(\dfrac{x}{\lambda_1(t)}\right)\big(\sqrt{\beta}\rho-\eta\big)(t,x-\upsilon_- t)\Phi'\left(\dfrac{x}{\lambda_1(t)}\right)\Phi'\left(\dfrac{x}{\lambda_2(t)}\right)
\\ & \quad \,+\dfrac{\theta\lambda_2'(t)}{\mu(t)\lambda_2(t)}\int \left(\dfrac{x}{\lambda_2(t)}\right)\big(\sqrt{\beta}\rho-\eta\big)(t,x-\upsilon_- t)\Phi\left(\dfrac{x}{\lambda_1(t)}\right)\Phi''\left(\dfrac{x}{\lambda_2(t)}\right)
\\ & \,=:\mathrm{I}+\mathrm{II}+\mathrm{III}+\mathrm{IV}+\mathrm{V}.
\end{align*}
Then, from Lemma \ref{virial_j1} we infer that \[
\dfrac{1}{\mu(t)\lambda_1(t)}\int\big\vert\psi(t,x-\upsilon_-t)\big\vert^2\Phi'\left(\dfrac{x}{\lambda_1(t)}\right)\Phi'\left(\dfrac{x}{\lambda_2(t)}\right)dx=\mathcal{F}_-(t).
\]
Hence, the problem is reduced to prove that we can integrate $\mathcal{F}_-(t)$ on $(0,+\infty)$. In fact, first of all notice that, from  Corollary \ref{uniform_bound_mod_functionals}  we infer  that \[
\left\vert \int_0^{+\infty}\mathrm{I}(t)dt\right\vert\lesssim\limsup_{t\to+\infty}\big\vert\mathcal{J}_1(t)-\mathcal{J}_1(0)\big\vert<+\infty.
\]
Moreover, from Lemma \ref{uniform_bound_energy} as well as 
the explicit definitions of $\mu(t)$ and $\lambda_2(t)$, it 
immediately follows that $\mathrm{II}\in L^1(\R_+)$. On the 
other hand, from Lemma \ref{uniform_bound_energy} along with
H\"older inequality, we can bound $\mathrm{III}(t)$ by 
\[
\big\vert\mathrm{III}(t)\big\vert \lesssim \dfrac{\mu'(t)\Vert \Phi'(\tfrac{\cdot}{\lambda_2(t)})\Vert_{L^2}}{\mu^2(t)}\lesssim\dfrac{\lambda_2^{1/2}\mu'(t)}{\mu^2(t)}\lesssim \dfrac{1}{(\kappa+t)\log (\kappa+t)\log\log^{3/2}(\kappa+t)}\in L^1(\R_+).
\]
In the same fashion, applying Lemma 
\ref{uniform_bound_energy} and H\"older inequality, 
we can bound 
$\vert \mathrm{IV}(t)\vert$ and 
$\vert\mathrm{V}(t)\vert$ pointwisely by integrable function 
as  \begin{align*}
\big\vert \mathrm{IV}(t)\big\vert&\lesssim \dfrac{\lambda_1^{1/2}(t)\lambda_1'(t)}{\mu(t)\lambda_1(t)}\lesssim \dfrac{1}{(\kappa+t)\log (\kappa+t)\log\log^{2}(\kappa+t)}\in L^1(\R_+),
\\ \big\vert\mathrm{V}(t)\big\vert &\lesssim \dfrac{\lambda_2^{1/2}(t)\lambda_2'(t)}{\mu(t)\lambda_2(t)}\lesssim \dfrac{1}{(\kappa+t)\log (\kappa+t)\log\log^{3/2}(\kappa+t)}\in L^1(\R_+).
\end{align*}
Therefore, gathering all the above inequalities, 
we conclude the proof of \eqref{time_integrability_l2} 
in the case of $\upsilon_-$. 
\par Notice that the same proof 
(up to trivial modifications) also works for $\upsilon_+$. 
The proof is complete.
\end{proof}

\begin{rem}
The proof of the previous proposition does not depend on the 
value of $\alpha\in\R$, and hence, this concludes the first 
inequality  in Theorem \ref{MT1}.
\end{rem}
\subsection{Time Integrability of the full solution}
In this section we seek to extend the analysis to the full solution, that is, to include the corresponding integral terms associated to $(\psi_x,\rho,\eta)$. From now on we split the analysis in two cases concerning the values of $\alpha\in\R$.

\subsubsection{Case $\alpha\neq 0$}
In order to take advantage of the previous analysis, we consider a different version of the modified momentum functional adapted to this region. More specifically, we define modified momentum functional adapted to the characteristics $x-\upsilon_\pm t$, that is,
\begin{align*}
\widetilde{\mathcal{I}}_\pm(t)&:=\dfrac{1}{\mu(t)}\mathrm{Im}\int_\R\psi(t,x-\upsilon_\pm t)\overline{\psi_x}(t,x-\upsilon_\pm t)\Phi\left(\dfrac{x}{\lambda_1(t)}\right)dx
\\ & \,\quad -\dfrac{\theta}{\mu(t)}\int_\R\rho(t,x-\upsilon_\pm t)\eta(t,x-\upsilon_\pm t)\Phi\left(\dfrac{x}{\lambda_1(t)}\right)dx.
\end{align*}
Notice that, as a direct consequence of Lemma \ref{virial_identity_lemma}, we have the following identity \begin{align}\label{dt_tilde_di}
\dfrac{d}{dt}\widetilde{\mathcal{I}}_\pm(t)&=\dfrac{2\omega}{\mu\lambda_1}\int \vert\psi_x\vert^2\Phi'-\dfrac{\omega}{2\mu\lambda_1^3}\int \vert\psi\vert^2\Phi'''+\dfrac{1}{2\mu\lambda_1 }\int \eta^2\Phi'+\dfrac{\beta}{2\mu\lambda_1 }\int \rho^2\Phi' \nonumber
	\\ & \quad  +\dfrac{3\gamma q}{4\mu\lambda_1}\int \vert\psi\vert^4\Phi'-\dfrac{\alpha}{\mu\lambda_1 }\int \rho\eta\Phi'+\dfrac{\gamma}{\mu\lambda_1 }\int \big(\eta-\tfrac{\alpha}{2}\rho\big)\vert\psi\vert^2\Phi'-\dfrac{\mu'}{\mu^2}\int\rho\eta\Phi\nonumber
	\\ &\quad +\dfrac{\mu'}{\mu^2}\mathrm{Im}\int\psi\overline{\psi_x}\Phi+\dfrac{\lambda_1'}{\mu\lambda_1}\mathrm{Im}\int\left(\dfrac{x}{\lambda_1}\right)\psi\overline{\psi}_x\Phi'-\dfrac{\lambda_1'}{\mu\lambda_1 }\int\left(\dfrac{x}{\lambda_1}\right)\rho\eta\Phi'
	\\ & \quad \pm\dfrac{\sqrt{\beta}\pm\alpha}{\mu\lambda_1}\mathrm{Im}\int \psi\overline{\psi}_x\Phi'\mp\dfrac{\sqrt{\beta}\pm\alpha}{\mu\lambda_1}\int \rho\eta\Phi' \nonumber,
\end{align}
where we have used the fact that, since $\Phi$ is real-valued, we have \begin{align*}
-\mathrm{Im}\int \psi\overline{\psi}_{xx}\Phi\left(\dfrac{x}{\lambda_1}\right)&=\dfrac{1}{\lambda_1}\mathrm{Im}\int \psi\overline{\psi}_x\Phi'\left(\dfrac{x}{\lambda_1}\right).
\end{align*}
With all of this at hand, we are ready to prove the time integrability of the full solution in weighted spaces along these characteristics. The following proposition concludes the proof the Theorem \ref{MT1} in the case $\alpha\neq 0$.
\begin{prop}
Let $(\psi, \rho, \eta )\in C(\R, H^1\times L^2\times L^2)$ be any solution to system \eqref{zr_br_eq} emanating from an initial data $(\psi_0,\rho_0,\eta_0)\in H^1\times L^2\times L^2$. For $\lambda_1(t)$, $\mu(t)$ and $\upsilon_\pm$ defined as in \eqref{def_scalings_time}-\eqref{def_upsilon_pm}, we have the following two cases:
\begin{enumerate}
\item If $\alpha>0$, then the following holds
\begin{align*}
\int_0^{+\infty}\frac{1}{\mu(t)\lambda_1(t)}\int_{\R}\left(\vert\psi_x\vert^2+\vert \psi\vert^2+\rho^2+\eta^2\right)(t, x-\upsilon_- t) \sech^4\left(\frac x {\lambda_1(t)}\right)
dxdt&<+\infty.
\end{align*}
\item If $\alpha<0$, then the following holds
\begin{align*}
\int_0^{+\infty}\frac{1}{\mu(t)\lambda_1(t)}\int_{\R}\left(\vert\psi_x\vert^2+\vert\psi\vert^2+\rho^2+\eta^2\right)(t, x-\upsilon_+ t) \sech^4\left(\frac x {\lambda_1(t)}\right)
dxdt&<+\infty.
\end{align*}
\end{enumerate}
\end{prop}

\begin{proof} We shall proceed in a similar fashion as in Proposition \ref{integrability_l2_prop}. However, notice that in this case we have several quadratic terms with no definite sign. The idea is to use Proposition \ref{integrability_l2_prop} to deal with the terms involving $\vert\psi\vert^2$, and to absorb the crossed terms of the form $\rho\eta$ with the ones with $\rho^2$ and $\eta^2$. In fact, first of all notice that, thanks to Lemma \ref{uniform_bound_energy} and the explicit definitions of $\mu$ and $\lambda_1$, it is not difficult to see that \[
\dfrac{1}{\mu\lambda_1^3}\int\vert\psi\vert^2\Phi'''\in L^1(\R_+),\,\quad \dfrac{\mu'}{\mu^2}\int \rho\eta\Phi \in L^1(\R_+),\,\quad \dfrac{\mu'}{\mu^2}\mathrm{Im}\int\psi\overline{\psi}_x\Phi\in L^1(\R_+).
\]
Moreover, from Lemma \ref{uniform_bound_energy} we also infer that we can integrate $\widetilde{\mathcal{I}}'(t)$ on $(0,+\infty)$. On the other hand, by H\"older inequality, the explicit definitions of $\mu$ and $\lambda_1$, as well as Lemma \ref{uniform_bound_energy}, we obtain
\begin{align*}
&\dfrac{\lambda_1'}{\mu\lambda_1}\left\vert\mathrm{Im}\int\left(\dfrac{x}{\lambda_1}\right)\psi\overline{\psi}_x\Phi'\right\vert\lesssim\dfrac{\lambda_1'\Vert \psi\Vert_{L^\infty_tH^1_x}^2\Vert (\cdot)\Phi'\Vert_{L^\infty}}{\mu\lambda_1}\in L^1(\R_+), \qquad \hbox{and},
\\ &\dfrac{\lambda_1'}{\mu\lambda_1}\left\vert\int\left(\dfrac{x}{\lambda_1}\right)\rho\eta\Phi'\right\vert\lesssim \dfrac{\lambda_1'\Vert \rho\Vert_{L^\infty_tL^2_x}\Vert \eta\Vert_{L^\infty_tL^2_x}\Vert(\cdot)\Phi'\Vert_{L^\infty}}{\mu\lambda_1}\in L^1(\R_+).
\end{align*}
Now, for the term involving $\vert\psi\vert^4$ in \eqref{dt_tilde_di}, we use Proposition \ref{integrability_l2_prop}, Lemma \ref{uniform_bound_energy} as well as Sobolev embedding, from where we get \begin{align}\label{bound_psi4}
\dfrac{1}{\mu\lambda_1}\int \vert\psi\vert^4\Phi'\lesssim  \dfrac{1}{\mu\lambda_1}\Vert \psi\Vert_{L^\infty_tH^1_x}^2\int \vert\psi\vert^2\Phi'\in L^1(\R_+).
\end{align}
Besides, from Young inequality for products, it is not difficult to see that
\[
\dfrac{2\omega}{\mu\lambda_1}\int\vert\psi_x\vert^2\Phi'\pm\dfrac{(\sqrt\beta\pm\alpha)}{\mu\lambda_1}\mathrm{Im}\int \psi\overline{\psi}_x\Phi'\geq \dfrac{\omega}{\mu\lambda_1}\int\vert\psi_x\vert^2\Phi'-\dfrac{(\sqrt{\beta}-\alpha)^2}{4\omega\mu\lambda_1}\int\vert\psi\vert^2\Phi'.
\]
Notice that the last term in the right-hand side above, belongs to $L^1(\R_+)$, thanks to Proposition \ref{integrability_l2_prop}. Also, from Young inequality for products again, we additionally infer that \begin{align}\label{varepsilon_star_12}
\dfrac{\gamma}{\mu\lambda_1}\int\big(\eta-\tfrac{\alpha}{2}\rho\big)\vert\psi\vert^2\Phi'\geq -\dfrac{\varepsilon^*_1}{\mu\lambda_1}\int\eta^2\Phi'-\dfrac{\varepsilon^*_2}{\mu\lambda_1}\int\rho^2\Phi'-\dfrac{K}{\mu\lambda_1}\int \vert\psi\vert^4\Phi'
\end{align}
where $\varepsilon^*_1,\varepsilon^*_2>0$ denote sufficiently small numbers that shall be fixed later. Here $K=K(\varepsilon_1^*,\varepsilon_2^*)>0$ is a large number, however, proceeding in the same fashion as in \eqref{bound_psi4} we infer that this term belongs to $L^1(\R_+)$ no matter what the value of $K$ is.

\medskip

Finally, it only remains to control the quadratic terms in $\rho$ and $\eta$. Unfortunately, due to the factor $\mp(\sqrt\beta\pm 2\alpha)\rho\eta$ appearing in \eqref{dt_tilde_di}, we cannot obtain the required positivity in both directions $\upsilon_\pm$ for the terms involving $\rho^2$, $\eta^2$ and $\rho\eta$. Thus, we split the analysis in two cases regarding the sign of $\alpha$.

\medskip

\textbf{Case $\alpha>0$:}  We aim to prove the following claim that provide us the required positivity: There exists two constants $c_1,c_2>0$, such that \begin{align}\label{claim1}
\dfrac{1}{2\mu\lambda}\int\eta^2\Phi'+\dfrac{\beta}{2\mu\lambda}\int\rho^2\Phi'+\dfrac{\sqrt{\beta}-2\alpha}{\mu\lambda}\int \rho\eta\Phi'\geq \dfrac{c_1}{\mu\lambda}\int \eta^2\Phi'+\dfrac{c_2}{\mu\lambda}\int\rho^2\Phi'.
\end{align}
In this case we shall take advantage of $\widetilde{\mathcal{I}}_-(t)$. In fact, first of all, notice that, if $\sqrt{\beta}-2\alpha=0$, then there is nothing to prove. Then, in the sequel we assume $\sqrt{\beta}-2\alpha\neq 0$. Indeed, consider the parameter $\varepsilon_1\in\R$ given by \[
\varepsilon_1:=\dfrac{1}{2}\left(1+\dfrac{\beta}{(\sqrt\beta-2\alpha)^2}\right)>0.
\]
Then, by using Young inequality for products, with parameter given by $\varepsilon_1$, we infer \begin{align*}
&\dfrac{1}{2\mu\lambda}\int\eta^2\Phi'+\dfrac{\beta}{2\mu\lambda}\int\rho^2\Phi'+\dfrac{\sqrt{\beta}-2\alpha}{\mu\lambda}\int \rho\eta\Phi'
\\ & \qquad \qquad \geq \dfrac{1}{2\mu\lambda}\big(\beta-(\sqrt\beta-2\alpha)^2\varepsilon_1 \big)\int \eta^2\Phi'+\dfrac{1}{2\mu\lambda}\big(1-\varepsilon_1^{-1}\big)\int\rho^2\Phi'
\end{align*}
Moreover, notice that, since $\beta-\alpha^2>0$ and $\alpha>0$, we infer that $\varepsilon_1>1$, and hence $1-\varepsilon_1^{-1}>0$. On the other hand, by direct computations we see that \[
\beta-(\sqrt\beta-2\alpha)^2\varepsilon_1>\beta-\dfrac{(\sqrt\beta-2\alpha)^2\beta}{(\sqrt\beta-2\alpha)^2}=0.
\]
Then, plugging the previous computations into \eqref{claim1}, we conclude the proof of the claim.

\medskip

\textbf{Case $\alpha<0$:} In this case we aim to prove the following claim that provide us the required positivity: There exists two constants $c_1,c_2>0$, such that \begin{align}\label{claim2}
\dfrac{1}{2\mu\lambda}\int\eta^2\Phi'+\dfrac{\beta}{2\mu\lambda}\int\rho^2\Phi'-\dfrac{\sqrt{\beta}+2\alpha}{\mu\lambda}\int \rho\eta\Phi'\geq c_1\int \eta^2\Phi'+c_2\int\rho^2\Phi'.
\end{align}
In contrast with the previous case, in this case we shall take advantage of $\widetilde{\mathcal{I}}_+(t)$. In fact, first of all, notice that, if $\sqrt{\beta}+2\alpha=0$, then there is nothing to prove. Thus, in the sequel we assume $\sqrt{\beta}+2\alpha\neq 0$. Indeed, define $\varepsilon_2\in\R$ as \[
\varepsilon_2:=\dfrac{1}{2}\left(1+\dfrac{\beta}{(\sqrt\beta+2\alpha)^2}\right)>0.
\]
Then, by using Young inequality for products, with parameter given by $\varepsilon_2$, we infer \begin{align*}
&\dfrac{1}{2\mu\lambda}\int\eta^2\Phi'+\dfrac{\beta}{2\mu\lambda}\int\rho^2\Phi'-\dfrac{\sqrt{\beta}+2\alpha}{\mu\lambda}\int \rho\eta\Phi'
\\ & \qquad \qquad \geq \dfrac{1}{2\mu\lambda}\big(\beta-(\sqrt{\beta}+2\alpha)^2\varepsilon_2\big)\int \eta^2\Phi'+\dfrac{1}{2\mu\lambda}\big(1-\varepsilon_2^{-1}\big)\int\rho^2\Phi'.
\end{align*}
Moreover, by using both $\beta-\alpha^2>0$ and $\alpha<0$, proceeding in exactly the same fashion as in the previous claim, we conclude that both factors in front of each of the integral on the right-hand side of the latter inequality are strictly positive. Thus, we conclude the proof of the claim.

\medskip

Finally, it only remains to set the definition of $\varepsilon_1^*$ and $\varepsilon_2^*$ in \eqref{varepsilon_star_12}. Notice that we have to give a different definition depending on the case $\alpha \gtrless 0$. In fact, from the above analysis it follows that it is enough to consider \begin{align*}
\varepsilon_{1,+}^*&:=10^{-10}\big(\beta-(\sqrt\beta-2\alpha)^2\varepsilon_1 \big) \qquad \varepsilon_{2,+}^*:=10^{-10}(1-\varepsilon_1^{-1}),
\\ \varepsilon_{1,-}^*&:=10^{-10}\big(\beta-(\sqrt\beta+2\alpha)^2\varepsilon_2 \big) \qquad \varepsilon_{2,-}^*:=10^{-10}(1-\varepsilon_2^{-1}),
\end{align*}
where $\varepsilon_{i,\pm}^*$ stands for the case where $\alpha$ is positive or negative respectively. Hence, we conclude the proof of the proposition.
\end{proof}

\subsubsection{Case $\alpha=0$} In the case when $\alpha=0$, we can give a much simpler and shorter proof. In fact, in this case, from Lemma \ref{virial_identity_lemma} we can easily deduce the following result.

\begin{prop}
Let $(\psi, \rho, \eta )\in C(\R, H^1\times L^2\times L^2)$ be any solution to system \eqref{zr_br_eq} emanating from an initial data $(\psi_0,\rho_0,\eta_0)\in H^1\times L^2\times L^2$. For $\lambda_1(t)$, $\mu(t)$ and $\upsilon_\pm$ defined as in \eqref{def_scalings_time}-\eqref{def_upsilon_pm}, the following inequality holds \[
\int_0^{+\infty}\dfrac{1}{\mu(t)\lambda_1(t)}\int_\R\Big(\vert\psi(t,x)\vert^4+\vert\psi_x(t,x)\vert^2+\eta^2(t,x)+\rho^2(t,x)\Big)\sech^2\left(\dfrac{x}{\lambda_1(t)}\right)dxdt<+\infty.
\]
\end{prop}

\begin{proof}
In fact, by using the standard modified momentum function 
$\mathcal{I}(t)$ (instead of $\widetilde{\mathcal{I}}(t)$ as before), we proceed in the same fashion as in the previous proposition, using Lemma 
\ref{virial_identity_lemma} and noticing that, under our 
current assumptions, $q=\gamma>0$, from where we infer that \[
\dfrac{3\gamma q}{ 4\mu\lambda_1}\int\vert\psi\vert^4\Phi'+\dfrac{1}{2\mu\lambda_1}\int\eta^2\Phi'+\dfrac{\gamma}{\mu\lambda_1}\int\eta\vert\psi\vert^2\Phi'> \dfrac{\gamma^2}{100\mu\lambda_1}\int \vert\psi\vert^4\Phi'+\dfrac{1}{100\mu\lambda_1}\int\eta^2\Phi'.
\]
Then, the proof follows by gathering the latter inequality with \eqref{di_dt} and recalling that $\lambda_1^{-3}\in L^1(\R_+)$. In order to avoid over-repeated computations we omit the details. 
\end{proof}

\medskip

\section{Decay in far field regions}\label{farfield_sec}

In this section we seek to prove pointwise decay in far field
regions by taking advantage of some suitable virial identities, as before. The analysis is similar (in spirit) to that shown in the previous section. However, in this case, the idea will be somewhat the opposite, in the sense that now the important terms shall come from the derivative of the weight $\Phi$, instead of the derivative of the solution, as in the previous section. To do so, we consider both the modified mass functional as well as the modified energy functional, which are given by (respectively)
\[
\mathcal M_\pm(t):=
\int_{\R}\Phi\left(\frac{\pm x+\zeta(t)}{\lambda(t)}\right) |\psi(t,x)|^2dx,
\]
\[
\mathcal E_\pm(t):=
\int_\R\Phi\left(\frac{\pm x+\zeta(t)}{\lambda(t)}
\right)\Big(\omega\vert\psi_x\vert^2%
+\tfrac{\gamma q}{2}\vert\psi\vert^4+\tfrac{\beta}{2}\rho^2
+\tfrac{1}{2}\eta^2+\tfrac{\gamma}{2}(2\eta-\alpha\rho) \vert\psi\vert^2-\alpha\rho\eta\Big)dx.
\]
Here, $\Phi$ stands for a smooth and bounded weight (not necessarily decaying at $\pm\infty$), which shall be completely different to the one chosen in the previous section (see \eqref{DefPhi1}-\eqref{DefPhi2} for the exact definition). Notice also that, in contrast with the modified mean functional, now we only require $\Phi$ belonging to $L^\infty(\R)$ in order for $\mathcal{M}_\pm$ and $\mathcal{E}_\pm$ to be well-defined and uniformly bounded (for solutions in the energy space). Additionally, in this case we define the scaling 
$\lambda(t)$ and the shift $\zeta(t)$ as 
\begin{align}\label{def_lambda_far_fields}
	\lambda(t):=(1+t)^{2+\delta}, 
	\quad {\zeta(t)\ge c_1 \lambda(t),
	\quad \zeta'(t)\ge c_2\lambda'(t),
	} \quad c_1, c_2>0, \ \delta>0.
\end{align}
In a similar spirit as in the previous section, the main motivation to consider these specific definitions of $\lambda$ and $\zeta$ is to obtain
\begin{align}\label{time_2_integrability_lambda_zeta}
	\dfrac{1}{\lambda},\dfrac{1}{\zeta} \in L^1(\R_+), \quad \hbox{however}\quad
	\frac{\lambda'}{\lambda},\frac{\zeta'}{\lambda} \not\in L^1(\R_+),
\end{align}
which shall allow us to neglect some bad terms (in some sense). We emphasize that, in this case, we are considering a scaling factor $\lambda(t)$  growing 
faster than linear (in contrast with the previous sections). 
This changes the behavior of some important terms (with 
respect to the above analysis) that we intend to take 
advantage of. 

\medskip

On the other hand, to simplify computations, we split the modified energy functional
$\mathcal E$ into the following functionals
\begin{align}
	&\mathcal E_{\pm,1}(t):=
	\int_\R\Phi\left(\frac{\pm x+\zeta(t)}{\lambda(t)}
	\right)\Big(\omega\vert\psi_x\vert^2%
	+\tfrac{\gamma q}{2}\vert\psi\vert^4\Big)dx,\label{FunctionalsFFR1}
	\\
	&
	\mathcal E_{\pm,2}(t):=
	\int_\R\Phi\left(\frac{\pm x+\zeta(t)}{\lambda(t)}
	\right)\Big(\tfrac{\beta}{2}\rho^2
	+\tfrac{1}{2}\eta^2
	+
	\tfrac{\gamma}{2}(2\eta-\alpha\rho) \vert\psi\vert^2
	-\alpha\rho\eta
	\Big)dx.\nonumber
	\end{align}

Before going further, let us compute the virial identities associated with our current functionals, that shall give us the fundamental information for the following analysis.
\begin{lem}
Let $(\psi,\rho,\eta)\in C(\R,H^1\times L^2\times L^2)$ be any solution to system \eqref{zr_br_eq}. Then, for all $t\in\R$, the following identity holds	
\begin{align}
	\frac{d}{dt}\mathcal M_\pm(t)
	=
	&
	-\frac{\lambda'}{\lambda}
	\int_{\R} \Phi'\left(\frac{\pm x+\zeta}{\lambda}\right)
	\left(\frac{\pm x+\zeta}{\lambda}\right)|\psi|^2dx
	+\frac{\zeta'}{\lambda}
	\int_{\R} \Phi'\left(\frac{\pm x+\zeta}{\lambda}\right)|\psi|^2dx\nonumber
\\ 	&\pm\frac{2}{\lambda}
	\mathrm{Im}\int_{\R} \Phi'\left(\frac{\pm x+\zeta}{\lambda}\right)
	\overline \psi \psi_x dx.\label{VirialFromMassFFR}
\end{align}
\end{lem}
\begin{proof}
The proof follows from direct computations. We omit this proof.
\end{proof}

\begin{lem}\label{VirialDecayAlongCurves}
Let $(\psi,\rho,\eta)\in C(\R,H^2\times H^1\times H^1)$ be any solution to system \eqref{zr_br_eq}. Then, for all $t\in\R$, the following identity holds,
\begin{align}	
	\frac{d}{dt} \mathcal E_\pm(t)
		&	
		=
		\frac{\zeta'}{\lambda}
		\int_\R
		\Phi'\left(\frac{\pm x+\zeta(t)}{\lambda(t)}\right)
		\Big(\omega\vert\psi_x\vert^2
		+\tfrac{\gamma q}{2}\vert\psi\vert^4\Big)dx\nonumber
		\\
		&
		\quad
		-\frac{\lambda'}{\lambda}
		\int_\R
		\Phi'\left(\frac{\pm x+\zeta(t)}{\lambda(t)}\right)
		\left(\frac{\pm x+\zeta(t)}{\lambda(t)}\right)
		\Big(\omega\vert\psi_x\vert^2
		+\tfrac{\gamma q}{2}\vert\psi\vert^4\Big)dx\nonumber
		\\
		&
		\quad 
		+\frac{\zeta'}{\lambda}
			\int_\R\Phi'\left(\frac{\pm x+\zeta(t)}{\lambda(t)}
			\right)\Big(\tfrac{\beta}{2}\rho^2
			+\tfrac{1}{2}\eta^2
			+
			\tfrac{\gamma}{2}(2\eta-\alpha\rho) \vert\psi\vert^2
			-\alpha\rho\eta
			\Big)dx	\label{VirialFromEnergyFFR}		
			\\
			&
			\quad 
			-
			 \frac{\lambda'}{\lambda}
			\int_\R
			\Phi'\left(\frac{\pm x+\zeta(t)}{\lambda(t)}
			\right)
			\left(\frac{\pm x+\zeta(t)}{\lambda(t)}
			\right)\Big(\tfrac{\beta}{2}\rho^2
			+\tfrac{1}{2}\eta^2
			+
			\tfrac{\gamma}{2}(2\eta-\alpha\rho) \vert\psi\vert^2
			-\alpha\rho\eta
			\Big)dx	\nonumber		
			\\
			&
			\quad
			\pm\frac{2\omega^2}{\lambda}
			\mathrm{Im}
			\int_\R
			\Phi'\left(\frac{\pm x+\zeta(t)}{\lambda(t)}\right)
			 \overline\psi_x \psi_{xx}
			dx
			\pm\frac{4 \gamma q \omega}{\lambda}
			\mathrm{Im}
			\int_\R\Phi'
			\left(\frac{\pm x+\zeta(t)}{\lambda(t)}\right)
			\vert\psi\vert^2 \overline \psi
			\psi_{x}
			dx\nonumber
			\\
			&
			\quad 
			\mp
			\frac \alpha {\theta\lambda}
			\int_\R 
			\Phi'\left(\frac{\pm x+ \zeta(t)}{\lambda(t)}\right)	
			\Big(
				\beta \vert \rho \vert ^2 + 
				\vert \eta \vert ^2
			\Big) 
			dx
			\pm 
			\frac{\beta+\alpha^2}{\theta\lambda}
			\int_\R 
			\Phi'\left(\frac{\pm x+ \zeta(t)}{\lambda(t)}\right)	
			 \rho\eta
			dx\nonumber
			\\
			&
			\quad 
			\pm
			\left(\frac{\gamma\beta}{\theta}
			+\frac{\gamma\alpha}{2\theta}\right)\frac1\lambda
			\int_\R 
			\Phi'\left(\frac{\pm x+ \zeta(t)}{\lambda(t)}\right)
			\rho\vert \psi \vert^2 dx
			\mp
			\frac32\frac{\gamma\alpha}{\theta\lambda}
			\int_\R 
			\Phi'\left(\frac{\pm x+ \zeta(t)}{\lambda(t)}\right)
			\eta\vert \psi \vert^2 
			dx\nonumber
			\\
			&
			\quad 
			\pm
			\frac{\gamma^2\alpha}{\theta\lambda}
			\int_\R 
			\Phi'\left(\frac{\pm x+ \zeta(t)}{\lambda(t)}\right)
			\vert \psi \vert^4
			dx
			\pm \frac{\gamma \omega}2 \frac 1 \lambda 
			\mathrm{Im} \int_\R 
			\Phi'\left(\frac{\pm x+ \zeta(t)}{\lambda(t)}\right)
			\left(
				2 \eta - \alpha \rho 
			\right)
			\overline \psi \psi_{x} dx. \nonumber
\end{align}
\end{lem}
\begin{proof}
First of all, in order to 
	simplify the computations, we split the derivative of $\mathcal E_\pm$ into the sum of the derivatives of $\mathcal{E}_{\pm,i}$, $i=1,2$, treating separately each of these functionals and then summing-up the corresponding results. In fact, directly differentiating $\mathcal{E}_{\pm,1}$, using that $(\psi,\rho,\eta)$ solves system 
	\eqref{zr_br_eq} and then performing several integration by parts, we obtain
\begin{align}
		\frac{d}{dt}\mathcal E_{\pm,1}(t)
		&
		=
		\frac{\zeta'}{\lambda}
		\int_\R
		\Phi'\left(\frac{\pm x +\zeta(t)}{\lambda(t)}\right)
		\Big(\omega\vert\psi_x\vert^2
		+\tfrac{\gamma q}{2}\vert\psi\vert^4\Big)dx\nonumber
		\\
		&
		\quad
		-\frac{\lambda'}{\lambda}
		\int_\R
		\Phi'\left(\frac{\pm x+\zeta(t)}{\lambda(t)}\right)
		\left(\frac{\pm x+\zeta(t)}{\lambda(t)}\right)
		\Big(\omega\vert\psi_x\vert^2
		+\tfrac{\gamma q}{2}\vert\psi\vert^4\Big)dx\label{Parte1VirialEnergia}
		\\
		&
		\quad
		\pm \frac{2\omega^2}{\lambda}
		\text{Im}
		\int_\R
		\Phi'\left(\frac{\pm x+\zeta(t)}{\lambda(t)}\right)
		 \overline\psi_x \psi_{xx}
		dx
		\pm \frac{4 \gamma q \omega}{\lambda}
		\text{Im}
		\int_\R\Phi'
		\left(\frac{\pm x+\zeta(t)}{\lambda(t)}\right)
		\vert\psi\vert^2 \overline \psi
		\psi_{x}
		dx\nonumber
		\\
		&
		\quad 
		+2\omega\gamma
		\text{Im}
		\int_\R
		\Phi\left(\frac{\pm x+\zeta(t)}{\lambda(t)}\right)
		\overline\psi_x\psi
		\left(\eta  -\tfrac12 \alpha \rho +q |\psi|^2
		\right)_x
		dx	\nonumber
		\\ & \quad +{ 2\gamma q \omega}
		\text{Im}
		\int_\R
		\Phi
		\left(\frac{\pm x+\zeta(t)}{\lambda(t)}\right)
		\left(\vert\psi\vert^2\right)_x
		\overline \psi
		\psi_{x}
		dx\nonumber 
		\\ & =: R_{\pm, 1}+R_{\pm, 2}+R_{\pm, 3}+R_{\pm, 4}
		+R_{\pm, 5}+ R_{\pm, 6}
		.\nonumber 
\end{align}	
We now proceed with $\mathcal E_{\pm, 2}$. 
In fact, in a similar fashion as before, directly 
differentiating $\mathcal{E}_{\pm,2}$,  using that 
$(\psi, \rho, \eta)$ solves \eqref{zr_br_eq}, and then 
performing several integration by parts, we get
		\begin{align}
		\frac d {dt}
		\mathcal E_{\pm,2}(t)
		 	&
			=
			\frac{\zeta'}{\lambda}
			\int_\R\Phi'\left(\frac{\pm x+\zeta(t)}{\lambda(t)}
			\right)\Big(\tfrac{\beta}{2}\rho^2
			+\tfrac{1}{2}\eta^2
			+
			\tfrac{\gamma}{2}(2\eta-\alpha\rho) \vert\psi\vert^2
			-\alpha\rho\eta
			\Big)dx
			\label{Parte2VirialEnergia}	
			\\
			&
			\quad 
			-
			 \frac{\lambda'}{\lambda}
			\int_\R
			\Phi'\left(\frac{\pm x+\zeta(t)}{\lambda(t)}
			\right)
			\left(\frac{\pm x+\zeta(t)}{\lambda(t)}
			\right)\Big(\tfrac{\beta}{2}\rho^2
			+\tfrac{1}{2}\eta^2
			+
			\tfrac{\gamma}{2}(2\eta-\alpha\rho) \vert\psi\vert^2
			-\alpha\rho\eta
			\Big)dx
			\nonumber		
			\\
			&
			\quad 
			\mp
			\frac \alpha {\theta\lambda}
			\int_\R 
			\Phi'\left(\frac{\pm x+ \zeta(t)}{\lambda(t)}\right)	
			\Big(
				\beta \vert \rho \vert ^2 + 
				\vert \eta \vert ^2
			\Big) 
			dx
			\pm 
			\frac{\beta+\alpha^2}{\theta\lambda}
			\int_\R 
			\Phi'\left(\frac{\pm x+ \zeta(t)}{\lambda(t)}\right)	
			 \rho\eta
			dx
			\nonumber
			\\
			&
			\quad 
			\pm 
			\left(\frac{\gamma\beta}{\theta}
			+\frac{\gamma\alpha}{2\theta}\right)\frac1\lambda
			\int_\R 
			\Phi'\left(\frac{\pm x+ \zeta(t)}{\lambda(t)}\right)
			\rho\vert \psi \vert^2 dx
			\mp
			\frac32\frac{\gamma\alpha}{\theta\lambda}
			\int_\R 
			\Phi'\left(\frac{\pm x+\zeta(t)}{\lambda(t)}\right)
			\eta\vert \psi \vert^2 
			dx
			\nonumber
			\\
			&
			\quad 
			\pm
			\frac{\gamma^2\alpha}{\theta\lambda}
			\int_\R 
			\Phi'\left(\frac{\pm x+ \zeta(t)}{\lambda(t)}\right)
			\vert \psi \vert^4
			dx
			\pm  \frac{\gamma \omega}2 \frac 1 \lambda 
			\text{Im} \int_\R 
			\Phi'\left(\frac{\pm x+ \zeta(t)}{\lambda(t)}\right)
			\left(
				2 \eta - \alpha \rho 
			\right)
			\overline \psi \psi_{x} dx
			\nonumber
			\\
			& 
			\quad
			+{\gamma \omega}
			\text{Im} \int_\R 
			\Phi\left(\frac{\pm x+ \zeta(t)}{\lambda(t)}\right)
			\left(
				 2\eta - \alpha \rho 
			\right)_x
			\overline \psi \psi_{x} dx.
			\nonumber
		\end{align}	
Finally, for the sake of simplicity let us define $R_{\pm,7}$ 
as the following quantity
		\[
		R_{\pm, 7}:=\gamma \omega 
		\text{Im} \int_\R 
		\Phi\left(\frac{\pm x+ \zeta(t)}{\lambda(t)}\right)
		\left(
		2\eta - \alpha \rho 
		\right)_x
		\overline \psi \psi_{x} dx.
		\]
Then, it is not difficult to see that with these definitions we have the relation 
$ R_{\pm, 5} =R_{\pm , 6}+ R_{\pm, 7}$. Therefore, summing up all the previous computations, and then using the above relation, we conclude the proof of \eqref{VirialFromEnergyFFR}.
\end{proof}

\subsection{Time integrability of the weighted $L^2$-norm}
\label{FFRFirstPart}
In this subsection we restrict ourselves to the simpler 
case of the time integrability of the weighted $L^2$-norm 
for the Schr\"odinger 
part of the solution $\psi(t,x)$. The integrability 
(and decay) of this weighted-norm is a fundamental part of 
the analysis since (as we shall see) it triggers the decay 
of the whole weighted energy norm. In fact, let us start by 
recalling the relation

	\begin{equation}\label{FFRMass1}
	\begin{aligned}
	&
	-\frac{d}{dt}\mathcal M_{\pm}(t)
	\pm\frac{2}{\lambda}
	\text{Im}\int_{\R} \Phi'\left(\frac{\pm x+\zeta}{\lambda}
	\right)
	\overline \psi \psi_x dx
	\\
	&\qquad
	=
	\frac{\lambda'}{\lambda}
	\int_{\R} \Phi'\left(\frac{\pm x+\zeta}{\lambda}\right)
	\left(\frac{\pm x+\zeta}{\lambda}\right)|\psi|^2dx
	-\frac{\zeta'}{\lambda}
	\int_{\R}
	 \Phi'\left(\frac{\pm x+\zeta}{\lambda}\right)|\psi|^2dx.
	 \end{aligned}
	\end{equation}
Then, notice that the left-hand side of the above relation is 
time-integrable on $\R_+$, provided that $\Phi' \in L^\infty(\R)$. 
In fact, if $\Phi'$ is bounded, then from H\"older 
inequality, Lemma \ref{uniform_bound_energy} and the fact that
\eqref{time_2_integrability_lambda_zeta} holds, we have
	\[	
	\left\vert \frac{2}{\lambda}
	\text{Im}\int_{\R} \Phi'\left(\frac{\pm x+\zeta}{\lambda}\right)
	\overline \psi \psi_x dx \right\vert
	\lesssim
	\frac{1}{\lambda} \in L^1(\R_+).
	\]
Motivated by the time-integrability above, as well as 
identity \eqref{FFRMass1}, we shall give a suitable 
definition for $\Phi\in C^\infty(\R)$ so that we are able to 
obtain a convenient sign-property in the right-hand side of 
\eqref{FFRMass1}. 

\medskip

To take 
advantage of the structure of the virial identities, 
from now on we consider $\Phi$ to be any non-increasing 
smooth function such
that it satisfies the following 
conditions
	\begin{equation}\label{DefPhi1}
	\big\{ \Phi(s)=1, \ s 	\le -1\big\},\quad  \big\{\Phi(s)=0, \ s \ge 0\big\} \quad \hbox{and} 
	\quad \{\Phi'\equiv -1 \hbox{ on }[-\tfrac{9}{10},-\tfrac{1}{10}]\}.
\end{equation}	
Notice that, as a particular consequence of its definition, 
we have the following inequalities
	\begin{equation}\label{DefPhi2}
	\forall s \in \R, \quad  \Phi'(s)\le 0 \quad 
	\hbox{and}\quad s\Phi'(s)\ge 0.
	\end{equation}
As already mentioned, we now focus in studying the right-hand 
side of \eqref{FFRMass1}. Notice that, thanks to 
\eqref{DefPhi1}-\eqref{DefPhi2} we infer that, for all 
$t\geq0$, the following sign-properties are satisfied
 	\[
 	\frac{\lambda'}{\lambda}
 	\int_{\R} \Phi'\left(\frac{\pm x+\zeta}{\lambda}\right)
 	\left(\frac{\pm x+\zeta}{\lambda}\right)|\psi|^2dx\geq0 \quad \hbox{and}\quad 
 	-\frac{\zeta'}{\lambda}
 	\int_{\R} \Phi'\left(\frac{\pm x+\zeta}{\lambda}\right)|\psi|^2dx\ge 0.
	 \]
	 
Consequently, due to the fact that the left-hand side of 
\eqref{FFRMass1}
is integrable in time, 
we can compute the time-integral over $\R_+$ and get
	\begin{equation}\label{IntegrableL2NormFFR}
	\int_0^\infty\left(
	\frac{\lambda'}{\lambda}
	\int_{\R} \Phi'\left(\frac{\pm x+\zeta}{\lambda}\right)
	\left(\frac{\pm x+\zeta}{\lambda}\right)|\psi|^2dx
	-\frac{\zeta'}{\lambda}
	\int_{\R} \Phi'\left(\frac{\pm x+\zeta}{\lambda}\right)|\psi|^2dx
	\right)dt <\infty.
	\end{equation}
Then, gathering this latter inequality with the sign-property above, we deduce in particular
	\begin{equation*}
	\int_0^\infty
	\frac{\zeta'(t)}{\lambda(t)}
	\int_{\R} \left|\Phi'\left(\frac{\pm x+\zeta(t)}{\lambda(t)}\right)\right|
	|\psi(t,x)|^2dxdt <\infty.
	\end{equation*}
\begin{rem} \label{Sequence}
Notice that, as a particular consequence of the latter inequality, recalling also that $\lambda^{-1}\zeta'\not\in L^1(\R_+)$, we infer the existence of a sequence of times $\{t_n\}_{n\in\R}$, satisfying $t_n \to \infty$,
such that
	\begin{equation}\label{FFR2}
	\lim_{n\to+\infty}
	\int_{\Omega(t_n)}\vert \psi(t_n,x)\vert^2dx = 0,
	\end{equation}
where the set $\Omega(t)$ can be defined, for example, as \begin{align}\label{omega_def}
\Omega(t):=\{x\in\R: \, \tfrac{1}{10}\lambda(t)+\zeta(t)\leq \vert x\vert \leq \tfrac{9}{10}\lambda(t)+\zeta(t)\}.
\end{align}
Moreover, notice that from the above \begin{align}\label{integrability_omega}
\int_0^\infty\dfrac{\lambda'(t)}{\lambda(t)}\int_{\Omega(t)}\vert \psi(t,x)\vert^2dx<+\infty \quad \hbox{and}\quad  \int_0^\infty\dfrac{\zeta'(t)}{\lambda(t)}\int_{\Omega(t)}\vert \psi(t,x)\vert^2dx<+\infty.
\end{align}
Hence, from now on, we can use properties \eqref{FFR2} and \eqref{integrability_omega} without depending on weights $\Phi$ satisfying \eqref{DefPhi1}. In particular, in the sequel we shall use \eqref{FFR2} for compactly supported weight functions encoding the same (or strictly contained) regions as in \eqref{omega_def}.
\end{rem}	


\subsection{Decay of the $L^2$-norm}\label{DecayL2Norm}
In this section we seek to prove the pointwise decay of 
$L^2$-norm of the Schr\"odinger component of the solution 
$\psi(t,x)$ restricted to the far-field regions. 
In fact, let us start by considering 
$\Psi \in C_0^\infty(\R)$ to be any non-negative function such that
\begin{equation}\label{FFRpsi1}
\text{supp} (\Psi) \subset \left[-\tfrac{3}{4}, 
-\tfrac{1}{4}\right] \quad \hbox{with} \quad \Psi\equiv 1 \hbox{ on } [-\tfrac{3}{5},-\tfrac{2}{5}].
\end{equation}
Notice that, in particular, this implies that 
{$\supp(\Psi)\subset \supp(\Phi')$}. Additionally, we assume 
that $\Psi$ satisfies the following pointwise properties
 	\begin{equation}\label{FFRpsi2}
	 \forall s \in \R, \quad \Psi(s)\lesssim |\Phi'(s)| 
	 \quad \hbox{and} \quad |\Psi'(s)|\lesssim |\Phi'(s)|.
	 \end{equation}
Now, we re-write the previous virial identity 
\eqref{VirialFromMassFFR} in terms of our new weight function $\Psi$ (instead of using $\Phi$). Then, using Lemma \ref{uniform_bound_energy} it is not difficult to see that
	\begin{align}
	&\left\vert \frac{d}{dt}\int_{\R}
	\Psi\left(\frac{x+\zeta}{\lambda}\right)|\psi|^2 dx\right\vert\nonumber	
	\\	& \quad \lesssim
	\frac{1}{\lambda}
	+\frac{\lambda'}{\lambda}\int_{\R}\left|
	\Psi'\left(\frac{\pm x+\zeta}{\lambda}\right)
	\left(\frac{\pm x+\zeta}{\lambda}\right)\right\vert\vert\psi\vert^2dx+\frac{\zeta'}{\lambda}
	\int_{\R}
	\left|\Psi'\left(\frac{\pm x+\zeta}{\lambda}\right)
	\right|
	|\psi|^2
	dx.
	\label{FFR1}
	\end{align}
Now, recall that, as stated in remark \ref{Sequence}, there exists a sequence of time $\{t_n\}_{n\in\N}$, satisfying $t_n \to \infty$, such that \eqref{FFR2} holds. As a consequence, we also have that 
 \[
	\lim_{n \to \infty}
	\left(
	\int_{\R} \Psi\left(\frac{x+\zeta}{\lambda} \right)
	|\psi|^2dx 
	\right)(t_n)= 0. 
\] 
Therefore, we can integrate both sides of \eqref{FFR1} in time over the interval $[t, t_n]$, and then take the limit $t_n\to \infty$, what lead us to 
\begin{align*}
	\int_{\R}
	\Psi\left(\frac{\pm x+\zeta}{\lambda}\right)|\psi|^2 dx & \lesssim
	\int_t^\infty
	\frac{ds}{\lambda(s)}+
	\int_t^\infty \frac{\zeta'}{\lambda} \int_{\R}
	\left|\Psi'\left(\frac{\pm x+\zeta}{\lambda}\right)
	\right|
	|\psi|^2
	dx ds
	\\ & \quad +\int_t^\infty \dfrac{\lambda'}{\lambda}\int_\R \left\vert \Psi'\left(\dfrac{\pm x+\zeta}{\lambda}\right)\left(\dfrac{\pm x+\zeta}{\lambda}\right)\right\vert\vert\psi\vert^2 dxds.
\end{align*}
Finally, by using both time-integrabilities in \eqref{integrability_omega}, 
we can take now the limit $t\to \infty$, from where we conclude that
	\[
	\lim_{t \to \infty}
	\int_{\R}
	\Psi\left(\frac{\pm x+\zeta}{\lambda}\right)|\psi|^2 dx
	=0.
	\]
	
\begin{rem}
Notice that the proof of the integrability (and subsequent decay) of the $L^2$-norm of $\psi$ also works for other definitions of $\lambda$ as well as other definitions of $\Phi$ and $\Psi$. For example, in the above analysis we have only used the fact that the scaling $\lambda(t)$ satisfies 
\[
\frac1\lambda	\in L^1(\R_+) \quad \hbox{ and } \quad  
\frac{\lambda'}{\lambda} \not\in L^1(\R_+). 
\]
In consequence, the proof still holds for any $\lambda$
with such property. As a result, we can take for example 
$\lambda(t)=ct^p$, for any $ p>2$ and any $c\in\R_+$, 
and then following the above computations we obtain the desired 
result. 
\end{rem}	

\subsection{Time-integrability of the full solution}
\label{IntegrabilityOfTheFullSolution}
In this section we seek to show the time integrability of the local energy norm for the remaining terms. Now, in order to make computations simpler, let us break down expression
\eqref{VirialFromEnergyFFR}, so that we can write the cleaner formula
\begin{equation}\label{VirialFromEnergyFFRReWritten}
-\frac d {dt} \mathcal E_\pm(t)+\mathcal R_\pm(t)=
-\mathfrak E_{\pm}(t).	
\end{equation}
More specifically, we define the functionals $\mathfrak{E}_\pm$ and 
$\mathcal{R}_\pm$ given by
\begin{align*}	
\mathfrak E_{\pm}&:= \frac{\zeta'}{\lambda}
		\int_\R
		\Phi'\left(\frac{\pm x+\zeta(t)}{\lambda(t)}\right)
		\Big(\omega\vert\psi_x\vert^2
		+\tfrac{\gamma q}{2}\vert\psi\vert^4\Big)dx\nonumber
		\\
		&
		\quad
		-\frac{\lambda'}{\lambda}
		\int_\R
		\Phi'\left(\frac{\pm x+\zeta(t)}{\lambda(t)}\right)
		\left(\frac{\pm x+\zeta(t)}{\lambda(t)}\right)
		\Big(\omega\vert\psi_x\vert^2
		+\tfrac{\gamma q}{2}\vert\psi\vert^4\Big)dx\nonumber
		\\
		&
		\quad 
		+\frac{\zeta'}{\lambda}
			\int_\R\Phi'\left(\frac{\pm x+\zeta(t)}{\lambda(t)}
			\right)\Big(\tfrac{\beta}{2}\rho^2
			+\tfrac{1}{2}\eta^2
			+
			\tfrac{\gamma}{2}(2\eta-\alpha\rho) \vert\psi\vert^2
			-\alpha\rho\eta
			\Big)dx	\label{VirialFromEnergyFFR}		
			\\
			&
			\quad 
			-
			 \frac{\lambda'}{\lambda}
			\int_\R
			\Phi'\left(\frac{\pm x+\zeta(t)}{\lambda(t)}
			\right)
			\left(\frac{\pm x+\zeta(t)}{\lambda(t)}
			\right)\Big(\tfrac{\beta}{2}\rho^2
			+\tfrac{1}{2}\eta^2
			+
			\tfrac{\gamma}{2}(2\eta-\alpha\rho) \vert\psi\vert^2
			-\alpha\rho\eta
			\Big)dx,	\nonumber		
\\ \mathcal R_\pm(t)
	&	
	:=
	\pm\frac{2\omega^2}{\lambda}
			\mathrm{Im}
			\int_\R
			\Phi'\left(\frac{\pm x+\zeta(t)}{\lambda(t)}\right)
			 \overline\psi_x \psi_{xx}
			dx
			\pm\frac{4 \gamma q \omega}{\lambda}
			\mathrm{Im}
			\int_\R\Phi'
			\left(\frac{\pm x+\zeta(t)}{\lambda(t)}\right)
			\vert\psi\vert^2 \overline \psi
			\psi_{x}
			dx\nonumber
			\\
			&
			\quad 
			\mp
			\frac \alpha {\theta\lambda}
			\int_\R 
			\Phi'\left(\frac{\pm x+ \zeta(t)}{\lambda(t)}\right)	
			\Big(
				\beta \vert \rho \vert ^2 + 
				\vert \eta \vert ^2
			\Big) 
			dx
			\pm 
			\frac{\beta+\alpha^2}{\theta\lambda}
			\int_\R 
			\Phi'\left(\frac{\pm x+ \zeta(t)}{\lambda(t)}\right)	
			 \rho\eta
			dx\nonumber
			\\
			&
			\quad 
			\pm
			\left(\frac{\gamma\beta}{\theta}
			+\frac{\gamma\alpha}{2\theta}\right)\frac1\lambda
			\int_\R 
			\Phi'\left(\frac{\pm x+ \zeta(t)}{\lambda(t)}\right)
			\rho\vert \psi \vert^2 dx
			\mp
			\frac32\frac{\gamma\alpha}{\theta\lambda}
			\int_\R 
			\Phi'\left(\frac{\pm x+ \zeta(t)}{\lambda(t)}\right)
			\eta\vert \psi \vert^2 
			dx\nonumber
			\\
			&
			\quad 
			\pm
			\frac{\gamma^2\alpha}{\theta\lambda}
			\int_\R 
			\Phi'\left(\frac{\pm x+ \zeta(t)}{\lambda(t)}\right)
			\vert \psi \vert^4
			dx
			\pm \frac{\gamma \omega}2 \frac 1 \lambda 
			\mathrm{Im} \int_\R 
			\Phi'\left(\frac{\pm x+ \zeta(t)}{\lambda(t)}\right)
			\left(
				2 \eta - \alpha \rho 
			\right)
			\overline \psi \psi_{x} dx
	\\
	&
	=:\mathcal R_{\pm, 1}+\mathcal R_{\pm, 2}
	+\cdots+\mathcal R_{\pm, 8}.
\end{align*}
On the other hand, notice that, since $(\psi,\rho,\eta)$ is a solution to system 
\eqref{zr_br_eq} belonging to the class 
$C(\R,H^2\times H^1\times H^1)$, then the energy associated
to this solution $ E\left(\psi(t), \rho(t), \eta(t)\right)$ 
is finite. This means that, because the weight $\Phi$ 
considered in 
the modified functional $\mathcal E$ is bounded, one has 
\[
 \int_{\R_+}
\frac d {dt} \mathcal E_\pm dt<\infty.
\]
Now, we treat the remaing term. 

\medskip

Let us consider {$\Phi\in C^\infty(\R)$} 
to be any non-increasing function such that \eqref{DefPhi1} 
holds, 
and hence, satisfying also 
\eqref{DefPhi2}. Now, we intend to bound term by term the 
right-hand side of \eqref{VirialFromEnergyFFR}. In fact, 
first, since $\Phi'$ is bounded, 
using Young inequality and Sobolev embbeding along with 
Lemma 
\ref{uniform_bound_energy}, we see that
	\[
	\vert \mathcal R_{\pm, 2}\vert
	\le 
	\frac{4 \gamma \omega \vert q \vert}{\lambda(t)}
	\Vert\psi\Vert_{L^\infty(\R)}^2
	\left(\Vert\psi\Vert_{L^2(\R)}^2
	+ \Vert \psi_x \Vert_{L^2(\R)}^2\right)
	\lesssim \frac{1}
	{\lambda(t)}  \in L^1(\R_+)
	\]	
Also, immediately from Proposition \eqref{uniform_bound_energy}, 
we obtain
	\[
	\vert\mathcal R_{\pm, 3}\vert \le 
	\frac{\vert\alpha\vert}{\theta\lambda(t)}
	\int_\R 
	\left|\Phi'\left(\frac{\pm x+\zeta(t)}{\lambda(t)}\right)\right|
	\left(
	\beta \vert \rho \vert^2 +\vert \eta \vert ^2
	\right)
	dx 
	\lesssim 
	\frac{1}{\lambda(t)}
	\in L^1(\R_+).
	\]
On the other hand, for $\mathcal R_{+,4}$, by using Young inequality for products and Lemma \ref{uniform_bound_energy} we get
	\[
	\vert\mathcal R_{\pm,4}\vert
	\le 
	\frac{\beta+\alpha^2}{2 \theta} \frac{1}{\lambda(t)} 
	\left(
		\Vert\rho(t)\Vert_{L^2(\R)}^2 
		+\Vert \eta(t) \Vert^2_{L^2(\R)}
	\right)
	\lesssim 
	\frac{1}{\lambda(t)} 
	\in L^1(\R_+).
	\]	
We point out that all the remaining terms $\mathcal{R}_{\pm,i}$, $i=5,...,8$, can be treated in the very same fashion as for the previous terms. In fact, from H\"older inequality, Sobolev embedding and then using Lemma \ref{uniform_bound_energy} in the resulting right-hand side, we deduce that
\[
\sum_{i=5}^8\vert \mathcal{R}_{\pm,i}\vert\lesssim 
	\frac{1}{\lambda(t)}
	\left(
	\Vert \eta(t) \Vert^2_{L^2(\R)}
	+ \Vert \rho(t) \Vert_{L^2(\R)}^2
	+\Vert \psi(t)\Vert_{H^1(\R)}^4\right)\lesssim\frac 1 {\lambda(t)} \in L^1(\R_+).
	\]
Finally, to complete the analysis regarding the time-integrability of $\mathcal{R}_\pm$, it only remains to consider  $\mathcal R_{\pm, 1}$. In order to do that, we first need to recall one of the main results proven in \cite{LiMa} which give us the polynomial growth of the $H^2$-norm of $\psi(t)$. In fact, from \cite[Proposition 1.1]{LiMa} we have that 
\begin{align}\label{ineq_lima}
\Vert \psi(t)\Vert_{H^2(\R)}\lesssim 1+\vert t\vert^{1^+}.	
\end{align}
As a consequence, recalling the explicit form of $\lambda(t)$ in \eqref{def_lambda_far_fields}, we conclude that \[
\dfrac{1}{\lambda(t)}\Vert\psi (t)\Vert_{H^2(\R)} \in L^1(\R_+)\] 
Therefore, by using H\"older inequality as well as Lemma \ref{uniform_bound_energy} and \eqref{ineq_lima}, 
we infer that 
\begin{align*}
	\vert \mathcal R_{\pm, 1}\vert 
	&
	\lesssim 
	\frac{1}{\lambda(t)}
	\Vert
	\psi_x(t)
	\Vert_{L^2(\R)}
	\Vert
	\psi_{xx}(t)
	\Vert_{L^2(\R)} 
	\in L^1(\R_+).
\end{align*}

In conclusion, we can integrate over $\R_+$ in time both sides of 
\eqref{VirialFromEnergyFFRReWritten}, from where we obtain 
\begin{equation}\label{CasiUltimaEcuacion}
\int_{\R_+}-\mathfrak E_\pm(t)dt<\infty.
\end{equation}
Now, let us break down this expression so that 
we can analyze the conflicting terms of $\mathfrak E_\pm(t)$
without sign. 
More specifically, we would like to absorbs or discard 
the part of the expression that does not constitute the 
weighted energy norm.  In fact, in similar fashion as in the proof of Proposition 
\ref{uniform_bound_energy}, we have that 
	\[
	\tfrac{\beta}{2}\rho^2
	+\tfrac{1}{2}\eta^2
	-\alpha\rho\eta
	\geq \dfrac{\beta-\alpha^2}{4}
	\vert \rho\vert^2
	+\dfrac{\beta}{2(\beta+\alpha^2)}
	\vert\eta\vert^2>0,
	\]
where we used the fact that $\beta-\alpha^2>0$ and $\beta>0$.
Also, we have that 
	\begin{align*}
	\dfrac{\gamma}{2}\big(2\eta
	-\alpha\rho\big)\vert\psi\vert^2 & 
	\leq 
	\dfrac{\beta-\alpha^2}{16}\vert\rho
	\vert^2
	+\dfrac{\beta}{16(\beta+\alpha^2)}\vert 
	\eta\vert^2
	\\ & \qquad + 
	\left(\dfrac{\gamma^2(\beta+\alpha^2)}{8\beta}
	+\dfrac{2\alpha^2\gamma^2}{\beta-\alpha^2}
	\right)\vert\psi\vert^4.	
	\end{align*}
Gathering both equations above, we get 	
\begin{equation}\label{estimationBadTerms}	
\begin{aligned}
	\tfrac{\beta}{2}\rho^2
	+\tfrac{1}{2}\eta^2
	-\alpha\rho\eta
	+\dfrac{\gamma}{2}\big(2\eta
	-\alpha\rho\big)\vert\psi\vert^2
	& 
	\geq 
	\dfrac{3(\beta-\alpha^2)}{16}\vert\rho
	\vert^2
	+\dfrac{7\beta}{16(\beta+\alpha^2)}\vert 
	\eta\vert^2
	\\ & \qquad -
	\left(\dfrac{\gamma^2(\beta+\alpha^2)}{8\beta}
	+\dfrac{2\alpha^2\gamma^2}{\beta-\alpha^2}
	\right)\vert\psi\vert^4.	
	\end{aligned}
\end{equation}

Finally, to deal with the remaining uncontrolled terms 
(involving the $L^4$-norm of $\psi$), 
we 
make use of \eqref{IntegrableL2NormFFR}. Indeed, 
notice that by Sobolev embedding,
	\begin{align*}
	&
	\int_\R
	\left[
	\frac{\lambda'}{\lambda}
	\Phi'\left(\frac{\pm x+\zeta(t)}{\lambda(t)}\right)
	\left(\frac{\pm x+\zeta(t)}{\lambda(t)}\right)
	-
	\frac{\zeta'}{\lambda}	
	\Phi'\left(\frac{\pm x+\zeta(t)}{\lambda(t)}\right)
	\right]
	\vert\psi\vert^4dx
	\\
	& 
	\le 
	\Vert \psi\Vert_{H^1(\R)}^2
	\int_\R
	\left[
	\frac{\lambda'}{\lambda}
	\Phi'\left(\frac{\pm x+\zeta(t)}{\lambda(t)}\right)
	\left(\frac{\pm x+\zeta(t)}{\lambda(t)}\right)
	-
	\frac{\zeta'}{\lambda}	
	\Phi'\left(\frac{\pm x+\zeta(t)}{\lambda(t)}\right)
	\right]
	\vert\psi\vert^2dx.
\end{align*}	
Then, thanks to \eqref{IntegrableL2NormFFR}, we can conlude 
that
\[
	\int_\R
	\left[
	\frac{\lambda'}{\lambda}
	\Phi'\left(\frac{\pm x+\zeta(t)}{\lambda(t)}\right)
	\left(\frac{\pm x+\zeta(t)}{\lambda(t)}\right)
	-
	\frac{\zeta'}{\lambda}	
	\Phi'\left(\frac{\pm x+\zeta(t)}{\lambda(t)}\right)
	\right]
	\vert\psi\vert^4dx	
\]
is integrable in time over $\R_+$. Therefore, one can write the following

\begin{align}	
&
\int_\R
	\left[
	\frac{\lambda'}{\lambda}
	\Phi'\left(\frac{\pm x+\zeta(t)}{\lambda(t)}\right)
	\left(\frac{\pm x+\zeta(t)}{\lambda(t)}\right)
	-
	\frac{\zeta'}{\lambda}	
	\Phi'\left(\frac{\pm x+\zeta(t)}{\lambda(t)}\right)
	\right]
	\omega\vert\psi_x\vert^2dx
	\nonumber	
\\
&
+\int_\R
	\left[
	\frac{\lambda'}{\lambda}
	\Phi'\left(\frac{\pm x+\zeta(t)}{\lambda(t)}\right)
	\left(\frac{\pm x+\zeta(t)}{\lambda(t)}\right)
	-
	\frac{\zeta'}{\lambda}	
	\Phi'\left(\frac{\pm x+\zeta(t)}{\lambda(t)}\right)
	\right]
	\tfrac{3(\beta - \alpha^2)}{16}\vert\rho\vert^2dx	
	\label{IntegrabilityOfVirialTermsFFR}
\\
&
+\int_\R
	\left[
	\frac{\lambda'}{\lambda}
	\Phi'\left(\frac{\pm x+\zeta(t)}{\lambda(t)}\right)
	\left(\frac{\pm x+\zeta(t)}{\lambda(t)}\right)
	-
	\frac{\zeta'}{\lambda}	
	\Phi'\left(\frac{\pm x+\zeta(t)}{\lambda(t)}\right)
	\right]
	\tfrac{7 \beta}{16 (\beta-\alpha^2)}\vert\eta\vert^2dx
	\nonumber	
\\
&
\quad 
	\le - \mathfrak{E}_{\pm} 
	+ K \int_\R
	\left[
	\frac{\lambda'}{\lambda}
	\Phi'\left(\frac{\pm x+\zeta(t)}{\lambda(t)}\right)
	\left(\frac{\pm x+\zeta(t)}{\lambda(t)}\right)
	-
	\frac{\zeta'}{\lambda}	
	\Phi'\left(\frac{\pm x+\zeta(t)}{\lambda(t)}\right)
	\right]
	\vert\psi\vert^4dx,	
	\nonumber
\end{align}	
where $K$ is the absolute value of a constant depending on 
$\beta, \alpha, \gamma, q$. This way, we conclude that the 
right-hand side of the inequality above can be integrated in time over $\R_+$. 
Moreover, since 
\[
\Phi'(s)s\ge 0 \quad \text{and} \quad 
\Phi'(s)\le 0 \quad 
\forall s \in \R, 	
\]
we also infer that 
\[
\int_{\R_+}
\frac{\zeta'}{\lambda}
	\int_\R
	\left|
	\Phi'\left(\frac{\pm x+\zeta(t)}{\lambda(t)}\right)
	\right|
	\left(
	\vert\psi_x\vert^2
	+\vert\rho\vert^2	
	+\vert\eta\vert^2\right)dx	
dt<\infty.
\]
\begin{rem}
We notice that, thanks to the fact that 
$\frac{\zeta'}\lambda \not\in L^1(\R_+)$, one 
infers that there exists a sequence $\{t_n\}$, with
$\{t_n\}\to \infty$, such that
\begin{equation}\label{UltimoLimite}
	\left(
		\int_\R
	\left|
	\Phi'\left(\frac{\pm x+\zeta(t_n)}{\lambda(t_n)}\right)
	\right|
	\left(
	\vert\psi_x\vert^2
	+\vert\rho\vert^2	
	+\vert\eta\vert^2\right)dx
	\right)(t_n)
	\to 0, \text{ as } t_n \to \infty.
\end{equation}	

\end{rem}

\subsection{Decay of the full solution}

Finally, in this subsection we devote ourselves to prove 
decay of solutions in the energy space along the curves $\pm \zeta$. The idea is the same as for the decay of the $L^2$-norm in Subsection \ref{DecayL2Norm}. We proceed by 
taking a convenient weight $\Psi$ such that \eqref{FFRpsi1} 
holds. Then, we have that $\supp(\Psi)\subset \supp(\Phi')$ and 
\eqref{FFRpsi2} is satisfied. Next, we consider the virial 
identity \ref{VirialDecayAlongCurves} with the weight 
$\Psi$ instead 
of $\Phi$. Thus, taking into account the previous estimations
stated in Subsection \ref{IntegrabilityOfTheFullSolution} 
along with the pointwise properties 
\eqref{FFRpsi1}-\eqref{FFRpsi2}, we 
have that 
\begin{align}
&\frac{d}{dt} 
	\int_\R\Psi\left(\frac{\pm x+\zeta(t)}{\lambda(t)}
	\right)\Big(\omega\vert\psi_x\vert^2%
	+\tfrac{\gamma q}{2}\vert\psi\vert^4+\tfrac{\beta}{2}\rho^2
	+\tfrac{1}{2}\eta^2+\tfrac{\gamma}{2}(2\eta-\alpha\rho) \vert\psi\vert^2-\alpha\rho\eta\Big)dx
	\nonumber
	\\	
	&	
	\quad 
		\le
		\frac{\zeta'}{\lambda}
		\int_\R
		\Psi'\left(\frac{\pm x+\zeta(t)}{\lambda(t)}\right)
		\Big(\omega\vert\psi_x\vert^2
		+\tfrac{\gamma q}{2}\vert\psi\vert^4\Big)dx
		\nonumber
		\\
		&
		\qquad
		-\frac{\lambda'}{\lambda}
		\int_\R
		\Psi'\left(\frac{\pm x+\zeta(t)}{\lambda(t)}\right)
		\left(\frac{\pm x+\zeta(t)}{\lambda(t)}\right)
		\Big(\omega\vert\psi_x\vert^2
		+\tfrac{\gamma q}{2}\vert\psi\vert^4\Big)dxdt
		\nonumber
		\\
		&
		\qquad 
		+\frac{\zeta'}{\lambda}
			\int_\R\Psi'\left(\frac{\pm x+\zeta(t)}{\lambda(t)}
			\right)\Big(\tfrac{\beta}{2}\rho^2
			+\tfrac{1}{2}\eta^2
			+
			\tfrac{\gamma}{2}(2\eta-\alpha\rho) \vert\psi\vert^2
			-\alpha\rho\eta
			\Big)dx
			\nonumber
			\\
			&
			\qquad 
			-
			 \frac{\lambda'}{\lambda}
			\int_\R
			\Psi'\left(\frac{\pm x+\zeta(t)}{\lambda(t)}
			\right)
			\left(\frac{\pm x+\zeta(t)}{\lambda(t)}
			\right)\Big(\tfrac{\beta}{2}\rho^2
			+\tfrac{1}{2}\eta^2
			+
			\tfrac{\gamma}{2}(2\eta-\alpha\rho) 
			\vert\psi\vert^2
			-\alpha\rho\eta
			\Big)dx	\nonumber		
			\\
			& 
			\qquad 
			+\frac C{\lambda(t)}\left(1+\Vert \psi(t)\Vert_{H^2(\R)}
			\right), 
			\nonumber
\end{align}
where we recall that $\lambda^{-1}$ and $\lambda^{-1}\Vert \psi \Vert_{H^2}$ are both time-integrable in $\R_+$. Moreover, the whole right-hand side of the last inequality is 
integrable. Indeed, because $\Psi$ satisfies 
\eqref{FFRpsi1}-\eqref{FFRpsi2}
then \eqref{IntegrabilityOfVirialTermsFFR} along with Young 
inequality implies that 

\begin{align*}
	&
	\frac{\zeta'}{\lambda}
		\int_\R
		\Psi'\left(\frac{\pm x+\zeta(t)}{\lambda(t)}\right)
		\Big(\omega\vert\psi_x\vert^2
		+\tfrac{\gamma q}{2}\vert\psi\vert^4\Big)dx
		\\
		&
		-\frac{\lambda'}{\lambda}
		\int_\R
		\Psi'\left(\frac{\pm x+\zeta(t)}{\lambda(t)}\right)
		\left(\frac{\pm x+\zeta(t)}{\lambda(t)}\right)
		\Big(\omega\vert\psi_x\vert^2
		+\tfrac{\gamma q}{2}\vert\psi\vert^4\Big)dxdt
		\\
		&
		+\frac{\zeta'}{\lambda}
		\int_\R\Psi'\left(\frac{\pm x+\zeta(t)}{\lambda(t)}
		\right)\Big(\tfrac{\beta}{2}\rho^2
		+\tfrac{1}{2}\eta^2
		+
		\tfrac{\gamma}{2}(2\eta-\alpha\rho) \vert\psi\vert^2
		-\alpha\rho\eta
		\Big)dx
		\\
		&
		-
		\frac{\lambda'}{\lambda}
		\int_\R
		\Psi'\left(\frac{\pm x+\zeta(t)}{\lambda(t)}
		\right)
		\left(\frac{\pm x+\zeta(t)}{\lambda(t)}
		\right)\Big(\tfrac{\beta}{2}\rho^2
		+\tfrac{1}{2}\eta^2
		+
		\tfrac{\gamma}{2}(2\eta-\alpha\rho) 
		\vert\psi\vert^2
		-\alpha\rho\eta
		\Big)dx
		\\
		&
		\le C \left(- \mathfrak{E}_{\pm} 
		+ K \int_\R
		\left[
		\frac{\lambda'}{\lambda}
		\Phi'\left(\frac{\pm x+\zeta(t)}{\lambda(t)}\right)
		\left(\frac{\pm x+\zeta(t)}{\lambda(t)}\right)
		-
		\frac{\zeta'}{\lambda}	
		\Phi'\left(\frac{\pm x+\zeta(t)}{\lambda(t)}\right)
		\right]
		\vert\psi\vert^4dx\right).
\end{align*}
Now, recall that we have already shown in the previous section that the right-hand side of the above inequality is time-integrable in $\R_+$. Therefore, we conclude that there exists a time-integrable function $g:\R\to\R$ such that we can write \begin{align*}
	\left \vert
	\frac{d}{dt} 
		\int_\R\Psi\left(\frac{\pm x+\zeta(t)}{\lambda(t)}
		\right)\Big(\omega\vert\psi_x\vert^2%
		+\tfrac{\gamma q}{2}\vert\psi\vert^4+\tfrac{\beta}{2}
		\rho^2
		+\tfrac{1}{2}\eta^2+\tfrac{\gamma}{2}(2\eta-\alpha\rho) \vert\psi\vert^2-\alpha\rho\eta\Big)dx
		\right \vert \le 	
				g(t).
	\end{align*}
 Consequently, we are entitled to integrate over the time 
 interval $[t, t_n]$ and, because of
 \eqref{UltimoLimite}, taking $t_n\to \infty$, we get 
\begin{align*} 
	&\int_\R\Psi\left(\frac{\pm x+\zeta}{\lambda}
		\right)\Big(\omega\vert\psi_x\vert^2%
		+\tfrac{\gamma q}{2}\vert\psi\vert^4+\tfrac{\beta}{2}
		\rho^2
		+\tfrac{1}{2}\eta^2+\tfrac{\gamma}{2}(2\eta-\alpha\rho) 
		\vert\psi\vert^2-\alpha\rho\eta\Big)dx \le 	
		\int_t^\infty
		g(\tau)dt.
\end{align*}
Now, notice that, by using inequality \eqref{estimationBadTerms} we can
re-write the expression above in terms of the energy norm, 
as
\begin{align*}
	\int_\R\Psi\left(\frac{\pm x+\zeta(t)}{\lambda(t)}
		\right)\Big(\omega\vert\psi_x\vert^2%
		&  
		+\tfrac{\gamma (\beta-\alpha^2)}{16}
		\rho^2
		+\tfrac{7\beta}{16(\beta+\alpha^2)}\eta^2
		\Big)(t,x)dx
		\\
		&
		\lesssim 	
		\int_t^\infty
		g(\tau)dt 
		+\int_\R\Psi\left(\dfrac{\pm x+\zeta(t)}{\lambda}\right)\vert \psi(t)\vert^4dx.
	\end{align*}
Finally, notice that the latter integral involving $\vert\psi(t,x)\vert^4$ converges to zero as 
$t \to \infty$. In fact, this is a consequence of the decay of the 
$L^2$-norm \eqref{MT2Tesis1} and Lemma \eqref{uniform_bound_energy}, 
along with the Gagliardo-Nirenberg 
inequality, that allows us to bound the $L^4$-norm with the 
$H^1$-norm and $L^2$-norm. Then, to conclude, 
we take $t\to \infty$ in the latter inequality above, from where we obtain the decay 
\[
	\lim_{t\to \infty}
		\int_\R\Psi\left(\frac{\pm x+\zeta(t)}{\lambda(t)}
			\right)\Big(\vert\psi_x(t,x)\vert^2	+			\rho^2(t,x)+\eta^2(t,x)\Big)(t,x)dx	=0.
\]
The proof is complete.

\end{document}